\documentclass[pdflatex,sn-mathphys-num]{sn-jnl}%

\usepackage{graphicx}%
\usepackage{multirow}%
\usepackage{amsmath,amssymb,amsfonts}%
\usepackage{amsthm}%
\usepackage{mathrsfs}%
\usepackage[title]{appendix}%
\usepackage{xcolor}%
\usepackage{textcomp}%
\usepackage{manyfoot}%
\usepackage{booktabs}%
\usepackage{listings}%
\usepackage{float}            
\usepackage{placeins}         
\usepackage{etoolbox}         
\usepackage{geometry}         
\usepackage{csquotes}         
\usepackage{adjustbox}        

\usepackage{hyperref}
\usepackage[nameinlink]{cleveref}

\usepackage{enumitem}
\usepackage{mathtools}

\usepackage[ruled,vlined]{algorithm2e}

\theoremstyle{thmstyleone}
\newtheorem{theorem}{Theorem}
\newtheorem{lemma}{Lemma}
\newtheorem{assumption}{Assumption}
\newtheorem{proposition}{Proposition}

\newtheorem{formulation}{Formulation}

\theoremstyle{thmstylethree}

\crefname{enumi}{Step}{Steps}
\crefname{enumii}{Step}{Steps}
\crefname{enumiii}{Step}{Steps}
\crefname{enumiv}{Step}{Steps}

\begin{document}

\title{A Decomposition-Based Data-Driven Branch-and-Bound Approach for Black-Box Global Optimization}

\author[1]{\fnm{Damilola} \sur{Fasiku}}\email{dfasiku@ncsu.edu}
\author*[1]{\fnm{Wentao} \sur{Tang}}\email{wtang23@ncsu.edu}

\affil[1]{\orgdiv{Department of Chemical and Biomolecular Engineering},
\orgname{North Carolina State University},
\orgaddress{\city{Raleigh}, \state{North Carolina}, \country{USA}}}

\abstract{Many scientific and engineering applications require the
\textit{global optimization} of expensive \textit{black-box} functions. In \textit{high-dimensional} settings, however, achieving computational scalability is
a major challenge. To address this challenge, we propose a distributed spatial
\textit{branch-and-bound} framework that exploits partially separable structure
for the global optimization of box-constrained black-box problems. The proposed
framework adopts a \textit{two-layer} structure in which branching is performed
only on a lower-dimensional set of coupling variables in an outer layer.
Within each outer node, the remaining non-overlapping (private) variables,
which appear independently across blocks, have their domains partitioned and
explored in parallel within an inner layer, resulting in a naturally
parallelizable framework that is well suited for implementation on multi-core
CPU architectures. At each node, valid upper bounds are obtained via sampling,
with optional refinement using a local solver, while lower bounds are computed
using data-driven convex underestimators constructed over partitions of the
private variables. A residual-based correction and finite-sample validity
analysis provide conditions under which the partitionwise underestimators are
valid lower bounds, while the proposed branching and refinement mechanism
establishes convergence to an $\epsilon$-global optimum. Numerical studies on benchmark and
ODE-constrained parameter-estimation problems demonstrate improvements in both
sampling efficiency and computational performance relative to monolithic
approaches.}


\maketitle

\section{Introduction}\label{sec:intro}

\subsection{Background}

Many scientific and engineering applications involve black-box optimization problems in which the objective function can be evaluated, but its explicit mathematical representation is unavailable to the optimizer~\cite{audet2017,eason2016}. These problems arise in a wide range of applications, including machine learning, materials and molecular design, transportation networks, and engineering design~\cite{conn2009,larson2019,boukouvala2014,osorio2013}. The need to solve such problems in the absence of reliable and readily obtainable derivative information has driven significant interest in derivative-free optimization (DFO) over the past five decades. In the literature, derivative-free optimization methods can be broadly categorized into direct-search and model-based approaches. Direct-search methods optimize by iteratively sampling the objective function according to prescribed rules or heuristics, including algorithms such as the Nelder--Mead simplex method, pattern search, and mesh adaptive direct search (MADS)~\cite{rios2013,nelder1965,torczon1997,audet2006}. Model-based methods, on the other hand, construct and update surrogate models of the objective function to guide the search, including trust-region methods based on quadratic models and Bayesian optimization~\cite{powell2003,frazier2018}. While convergence to stationary points under suitable assumptions has been established for many of these methods, see for example~\cite{conn2009tr}, convergence to a global solution remains challenging. This limitation is particularly significant in applications where global optimality is essential, as local solutions may differ substantially from the global optimum, resulting in significant losses in performance or economic value~\cite{floudas2000,grossmann1996,locatelli2013}. Achieving global optimality typically requires more extensive exploration of the search space, thereby increasing the computational burden of black-box optimization. In high-dimensional settings, this burden is further exacerbated by the \textit{curse of dimensionality}, as sampling requirements grow exponentially with problem dimension~\cite{rios2013,ling2017}.

The foregoing discussion motivates the need for methods capable of providing guarantees of global optimality in black-box settings. In general, deterministic global optimization has been studied extensively in the context of nonconvex nonlinear programming and mixed-integer nonlinear programming (MINLP)~\cite{floudas2000}.  A variety of approaches have been developed, including Lipschitz optimization~\cite{jones1993}, interval analysis~\cite{hansen2003}, and spatial branch-and-bound~\cite{tawarmalani2013}. Among these, spatial branch-and-bound has emerged as one of the most successful frameworks for obtaining global optimality guarantees and forms the foundation of state-of-the-art solvers such as BARON, ANTIGONE, and SCIP~\cite{sahinidis1996,misener2014,achterberg2009}. The method operates by recursively partitioning the search space and computing valid upper and lower bounds over the resulting subregions, enabling systematic convergence to a global optimum through exhaustive exploration.

Despite its success in global optimization, extending spatial branch-and-bound
to black-box settings presents significant scalability challenges. In
data-driven implementations, the number of samples required to construct
regional lower bounds can grow rapidly with problem dimension, resulting in
substantial computational and memory requirements. This sampling burden is
further compounded by the potentially exponential growth in the number of
regions generated during branching. Furthermore, the efficiency of the
branch-and-bound procedure depends critically on the availability of tight
and computationally tractable lower bounds, which directly influence its
ability to prune regions effectively. Although powerful relaxation-based
techniques have been developed for nonlinear and mixed-integer nonlinear
optimization problems~\cite{floudas2000,tawarmalani2013}, these techniques
typically rely on explicit representations of the objective and constraints
to construct analytical relaxations~\cite{scott2011}. Such information is
generally unavailable in black-box settings, making these lower bounds
considerably more difficult to construct~\cite{boukouvala2016}. Consequently,
lower-bound construction and branching complexity remain major challenges in
developing scalable spatial branch-and-bound methods for black-box global
optimization.

Motivated by these challenges, this work proposes a spatial branch-and-bound algorithm for black-box problems with exploitable structure, specifically partially separable functions with overlapping variables. The key idea is to reduce branching complexity by performing branching only on a lower dimensional set of overlapping (coupling) variables in an outer layer, with upper bounds obtained via sampling and lower bounds constructed using a data-driven convex underestimator proposed in~\cite{zhai2022a}. For each node in this outer layer, the remaining non-overlapping (private) variables, which appear independently across blocks, have their domains partitioned and explored in parallel. This results in a naturally parallelizable framework that is well suited for implementation on multi-core CPU architectures. The class of problems considered in this work is formally defined next.

\subsection{Problem Description}

We consider the global optimization problem

\begin{equation}
\begin{aligned}
\min_{x \in \mathcal{X}} \quad & f(x)
\end{aligned}
\label{eq:main_problem}
\end{equation}

\noindent where $\mathcal{X} = [l,u] \subset \mathbb{R}^n$ is a bounded hypercube domain defined by lower and upper bounds $l,u \in \mathbb{R}^n$, and $f:\mathbb{R}^n \rightarrow \mathbb{R}$ is an objective function assumed to be continuous on $\mathcal{X}$, with derivative information unavailable to the optimizer. We focus on problems in which the objective function exhibits a partially separable structure with overlapping variables, such that

\begin{equation}
\begin{aligned}
f(x) = \sum_{k=1}^{K} f_k(x_{S_k})
\end{aligned}
\label{eq:separable_structure}
\end{equation}

\noindent where each $f_k : \mathbb{R}^{|S_k|} \rightarrow \mathbb{R}$ depends only on the variables indexed by $S_k \subseteq \{1,\dots,n\}$, and the index sets $S_k$ may intersect, thereby inducing coupling between the corresponding block functions.

\medskip
To exploit this structure, we distinguish between \textit{coupling variables}, 
which appear in multiple blocks, and \textit{private variables}, which appear 
in exactly one block. Let $s \in \mathbb{R}^{d_s}$ denote the vector of 
\textit{coupling variables}, and let $x_k \in \mathbb{R}^{d_k}$ denote the 
vector of \textit{private variables} associated with block $k$. The full 
decision vector may then be partitioned as

\begin{equation}
\begin{aligned}
x = (s, x_1, x_2, \dots, x_K).
\end{aligned}
\label{eq:variable_partition}
\end{equation}

\medskip

\noindent The goal is to compute a globally optimal solution $x^* \in \mathcal{X}$ that minimizes $f(x)$ while exploiting the underlying partially separable structure to improve computational efficiency.

\subsection{Relevant Literature}

Several effective derivative-free methods have been developed for the global optimization of black-box functions. Of particular note is SNOBFIT~\cite{huyer2008}, which addresses expensive and potentially noisy bound-constrained problems by combining domain partitioning with local quadratic models. It balances local improvement with exploration of sparsely sampled regions through complementary point-generation strategies. Similarly, the more recent Branch-and-Model (BAM) framework~\cite{ma2023} partitions the domain and constructs local surrogate models within selected regions to guide the search while maintaining global exploration. Its exploration mechanism yields asymptotic convergence to the global optimum under suitable assumptions. While these surrogate-based methods can efficiently identify high-quality solutions, their models are used to select promising evaluation points rather than constructed as valid lower bounds over individual regions. Consequently, they do not provide a finite-time optimality gap or support the rigorous exclusion of unexplored regions. Branch-and-bound frameworks provide these capabilities through valid regional lower bounds, making them particularly useful when a prescribed level of global optimality must be verified. Despite this advantage, branch-and-bound approaches remain relatively underexplored in black-box optimization.

Most directly related to the present work is the data-driven spatial branch-and-bound framework proposed by Zhai and Boukouvala~\cite{zhai2022a}, in which convex quadratic underestimators, which will be employed also in this work, are constructed from sampled data by solving a linear program that enforces underestimation at sampled points. This enables the incorporation of data-driven modeling within a spatial branch-and-bound framework for derivative-free global optimization. A recent extension~\cite{ravutla2026} incorporates Lipschitz continuity to improve underestimator validity and establish convergence guarantees under standard assumptions. In practice, these guarantees depend on accurate Lipschitz estimation and sufficient sampling density, which can be difficult to ensure in high-dimensional and expensive black-box settings. In addition, as with classical spatial branch-and-bound methods, scalability remains a fundamental challenge due to the up-to-exponential growth in branching complexity with increasing dimension.

We note that parallel branch-and-bound has been studied extensively, beginning with early work that distributed nodes or subtrees across processors using static or dynamic workload partitioning strategies~\cite{pekny1990,laursen1993,eckstein1994}. Subsequent developments introduced more sophisticated parallelization schemes, including master--worker frameworks, to improve load balancing and scalability in distributed environments~\cite{goux2002}. More recent efforts have focused on integrating parallel branch-and-bound within modern optimization solvers and computing architectures~\cite{shinano2018}. Despite these advances, most parallel branch-and-bound methods focus on distributing node evaluations across processors without modifying the underlying branching structure, and their performance can still be restricted by load imbalance, communication overhead, and variability in node-processing costs.

Generally, optimization problems with decomposable and overlapping-variable
structures have been studied through a variety of decomposition approaches.
Generalized Benders decomposition exploits such structure by separating
complicating variables into a master problem and representing the subordinate
problems through iteratively generated cuts~\cite{geoffrion1972}. Multicut and accelerated hybrid
multicut variants have further been developed to exploit identified problem
structure and decompose the original problem into a master problem and
independent subproblems~\cite{mitrai2022}. Related Lagrangian decomposition approaches relax coupling constraints to enable independent subproblem solution, while the alternating direction method of multipliers (ADMM) combines augmented-Lagrangian and decomposition ideas by alternating between primal-variable updates and a dual-variable update, enabling structured optimization problems to be solved in a distributed manner~\cite{boyd2011}. The ADMM framework has also been extended to nonconvex
constrained distributed optimization through an extra-layer architecture that
permits inexact solution of the local subproblems~\cite{tang2022}, and
subsequently to the derivative-free setting, where the local black-box
subproblems are solved inexactly using derivative-free trust-region
methods~\cite{fasiku2025}. Partially separable structure has also been exploited in derivative-free optimization through structured direct-search methods~\cite{porcelli2022}. These approaches exploit problem structure to improve computational efficiency but provide local rather than global optimization guarantees. Notably, in global optimization, restricting branching to a subset of variables is also not new. Related strategies arise in decomposition methods for angularly structured problems, including two-stage stochastic programs, where coupling variables are handled at a higher level and the remaining variables are treated in subordinate problems~\cite{cao2019,stinchfield2025}. The proposed approach similarly exploits decomposable structure to reduce branching complexity but accommodates general overlapping-variable structures in a fully derivative-free global optimization framework. Moreover, its lower-bounding scheme avoids solving nonconvex subproblems while explicitly enforcing the coupling relationships induced by shared variables.

\subsection{Summary of Contributions}

We propose a spatial branch-and-bound algorithm for black-box optimization
problems with partially separable structure and overlapping variables, designed
to improve scalability in high-dimensional settings. The main technical
contributions are as follows.

\textbf{1. Reduced-dimensional branching with parallel structured decomposition.}
We restrict branching to a lower-dimensional set of coupling variables and
handle the remaining private variables within each node via structured domain
partitioning in a parallelizable manner, reducing the effective branching
complexity.

\textbf{2. Data-driven lower bounding with finite-sample validity analysis.}
We develop a corrected data-driven convex underestimator for black-box functions
and establish conditions under which sampled residual information and geometric
sample quality yield a valid lower bound beyond the evaluated points. This
provides a theoretical basis for incorporating finite-sample data-driven
underestimation within spatial branch-and-bound.

\textbf{3. Convergence analysis and numerical validation.}
We establish convergence properties of the overall branch-and-bound framework
and demonstrate its effectiveness through numerical experiments on structured
high-dimensional test problems.

The remainder of this work is organized as follows.
Section~\ref{sec:prelim} reviews spatial branch-and-bound and data-driven convex
underestimation. Section~\ref{sec:algo} presents the proposed algorithm.
Section~\ref{sec:theoretical_analysis} establishes its theoretical properties.
Numerical examples are presented in Section~\ref{sec:numerical}, and conclusions
are given in Section~\ref{sec:conclusion}.

\section{Preliminaries}\label{sec:prelim}

\subsection{Spatial Branch-and-Bound}

We consider the global optimization problem in \eqref{eq:main_problem}. Spatial branch-and-bound is a deterministic framework for solving such problems to global optimality by partitioning the feasible region and computing bounds over the resulting subregions~\cite{tawarmalani2013,sahinidis1996}. The algorithm proceeds by recursively subdividing the domain $\mathcal{X}$ into smaller regions. Let $X^{(q)} \subseteq \mathcal{X}$ denote the region associated with node $q$ of the branch-and-bound tree. For each node, a lower bound $\mathrm{LB}^{(q)}$ on the optimal objective value over $X^{(q)}$ is computed, while an upper bound $\mathrm{UB}^{(q)}$ is obtained from feasible evaluations of the objective function. Let $\mathrm{UB}^\star$ denote the incumbent upper bound corresponding to the best feasible solution found so far. Nodes satisfying $\mathrm{LB}^{(q)} \geq \mathrm{UB}^\star$ cannot contain a globally optimal solution and may therefore be safely pruned. The remaining nodes are further partitioned and explored. The algorithm is terminated when the gap between the global upper and lower bounds falls below a prescribed optimality tolerance. A schematic illustration of this process is shown in Figure~\ref{fig:sbb_schematic}.
\begin{figure}[H]
\centering
\includegraphics[width=0.85\textwidth]{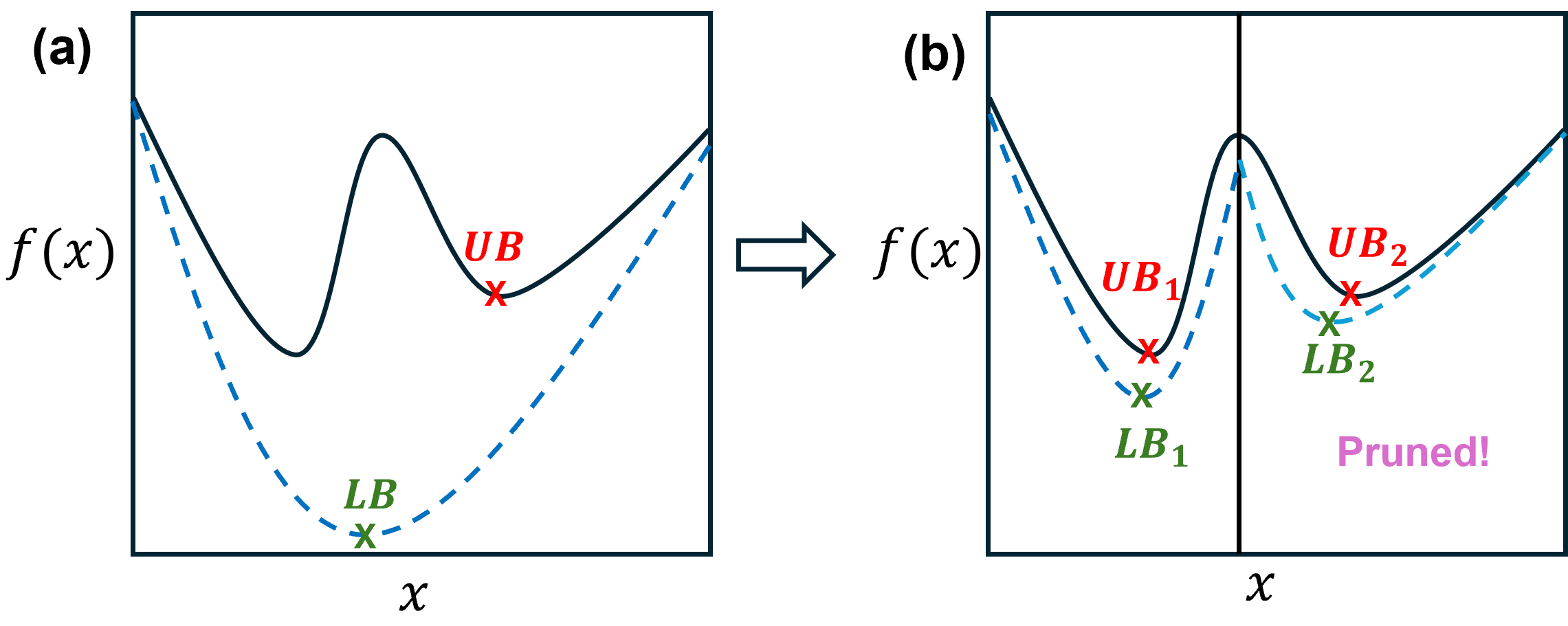}
\caption{Schematic illustration of spatial branch-and-bound. (a) Global lower and upper bounds over the full domain. (b) Partitioning of the domain into subregions with region-specific bounds used for pruning and branching.}
\label{fig:sbb_schematic}
\end{figure}

A fundamental component of spatial branch-and-bound is the construction of valid lower bounds. In deterministic global optimization, these bounds are typically obtained through convex relaxations of the original problem. Such relaxations replace a nonconvex problem by a tractable convex approximation whose optimal value is guaranteed to underestimate the true objective value over the region of interest. For functions with known analytical structure, McCormick
relaxations~\cite{mccormick1976} provide a systematic procedure for
constructing convex and concave relaxations by propagating relaxations through
the computational graph of the function. Solving the resulting convex relaxation yields a valid lower bound for the corresponding branch-and-bound node.

The effectiveness of spatial branch-and-bound depends critically on two components: the lower-bounding procedure and the branching strategy. The lower-bounding procedure determines the strength of the bounds used for pruning, while the branching strategy determines how the search space is partitioned and explored. In classical deterministic settings, both branching and bounding are performed directly in the full-dimensional decision space. Furthermore, the construction of relaxation-based lower bounds relies on explicit algebraic representations of the objective and constraint functions. In derivative-free settings, however, such information is unavailable, motivating the need for alternative lower-bounding mechanisms. Algorithm~\ref{alg:sbb} presents a conceptual description of the spatial branch-and-bound framework; practical implementations differ primarily in the choice of lower-bounding procedures, node-selection strategies, and branching rules.

\vspace{ 3mm}

\begin{algorithm}[!h]
\caption{Spatial Branch-and-Bound}
\label{alg:sbb}

\KwIn{Search domain $\mathcal{X}$, objective function $f$, and tolerance $\epsilon>0$}
\KwOut{Best feasible solution $x^\star$}

Select an initial feasible point $x^{(0)}\in\mathcal{X}$\;
Set $x^\star\leftarrow x^{(0)}$ and
$\mathrm{UB}^\star\leftarrow f(x^\star)$\;

Compute a valid lower bound $\mathrm{LB}^{(0)}$ over $\mathcal{X}$\;
Initialize the active-region set $\mathcal{L}\leftarrow\{\mathcal{X}\}$\;
Set $\mathrm{LB}^\star\leftarrow\mathrm{LB}^{(0)}$\;

\While{$\mathcal{L}\neq\emptyset$ and
$\mathrm{UB}^\star-\mathrm{LB}^\star>\epsilon$}{

Select the active region $X^{(q)}\in\mathcal{L}$ with the smallest
lower bound and remove it from $\mathcal{L}$\;

Obtain a feasible point $\hat{x}^{(q)}\in X^{(q)}$\;

\If{$f(\hat{x}^{(q)})<\mathrm{UB}^\star$}{
    $x^\star\leftarrow\hat{x}^{(q)}$\;
    $\mathrm{UB}^\star\leftarrow f(\hat{x}^{(q)})$\;
}

\eIf{$\mathrm{LB}^{(q)}\geq\mathrm{UB}^\star$}{
    Prune $X^{(q)}$\;
}{
    Partition $X^{(q)}$ into a collection of child regions $\mathcal{C}^{(q)}$\;
    Compute a valid lower bound for each child region\;
    $\mathcal{L}\leftarrow\mathcal{L}\cup\mathcal{C}^{(q)}$\;
}

\If{$\mathcal{L}\neq\emptyset$}{
    $\displaystyle
    \mathrm{LB}^\star
    \leftarrow
    \min_{X^{(j)}\in\mathcal{L}}\mathrm{LB}^{(j)}
    $\;
}

}

\Return{$x^\star$}\;

\end{algorithm}

\subsection{Data-Driven Convex Underestimation}
\label{sec:ddcu}

Convex underestimators are a key component of spatial branch-and-bound methods,
as they provide lower bounds over partitioned regions of the search space. For
a region $X\subseteq\mathcal{X}$, an underestimator $m$ is said to be valid on
$X$ if
\begin{equation}
m(x)\leq f(x),
\qquad \forall x\in X.
\label{eq:valid_underestimator}
\end{equation}
Consequently, minimizing a valid underestimator over $X$ provides a valid lower
bound on the minimum of $f$ over that region. In derivative-free settings,
where the functional form of the objective is unknown, such underestimators
must be constructed directly from sampled data. In this work, we adopt the
data-driven convex underestimation approach proposed by Zhai and
Boukouvala~\cite{zhai2022a}.

Given a set of sampled points $\{x_i, f(x_i)\}_{i=1}^{N}$, a convex quadratic model of the form
\begin{equation}
m(x)
=
\sum_{j=1}^{D} a_j x_j^2 + b_j x_j + c
\label{eq:ddcu_model}
\end{equation}
is constructed by solving the following linear program:
\begin{equation}
\label{eq:ddcu_lp}
\begin{aligned}
\min_{a,b,c}\quad &
\sum_{i=1}^{N}
\left(
f(x_i)-m(x_i)
\right)
\\
\text{s.t.}\quad &
m(x_i)\le f(x_i),
\qquad i=1,\ldots,N,
\\
&
a_j\ge0,
\qquad j=1,\ldots,D.
\end{aligned}
\end{equation}

\noindent The constraints enforce underestimation at all sampled points, while the nonnegativity constraints on the quadratic coefficients ensure convexity. Consequently, the resulting model is guaranteed to underestimate the sampled data, although it is not necessarily a valid underestimator over the entire domain. As the number and density of samples increase, however, the approximation quality generally improves and the model becomes increasingly representative of the underlying objective function. Also, as the region over which underestimation is performed becomes sufficiently small, the underestimator can become tight, as the underlying function $f$ becomes increasingly well approximated by a linear function locally.

An important feature of \eqref{eq:ddcu_model} is its diagonal quadratic structure. While more general convex quadratic models containing cross terms could potentially yield tighter approximations, the diagonal form is separable with respect to the individual variable components and therefore offers significant computational advantages. In particular, minimization of $f_{\mathrm{lb}}$ over box-constrained domains reduces to a collection of independent one-dimensional optimization problems, allowing lower bounds to be evaluated analytically and at negligible computational cost once the model coefficients have been obtained. These properties make the underestimator particularly attractive as a building block for scalable lower-bounding procedures. As will be shown in the next section, the resulting separability enables the efficient construction of reduced lower-bounding models and is useful in our proposed framework.

\section{Proposed Algorithm}\label{sec:algo}

\subsection{Overview of the Two-Level Framework}

In this work, we propose a distributed data-driven branch-and-bound method for solving the global optimization problem in \eqref{eq:main_problem}. The proposed approach exploits the partially separable structure in \eqref{eq:separable_structure}, where the objective function is expressed as a sum of overlapping block functions. To make this structure explicit, the variables are partitioned into two groups: 
coupling variables and private variables. Coupling variables appear in multiple 
block functions and induce dependencies across blocks, while private variables 
are associated with a single block function. Let $s$ denote the vector of 
coupling variables and let $x_k$ denote the vector of private variables associated 
with block $k$. Each block function can then be written as $f_k(s_k,x_k)$, where 
$s_k$ denotes the subvector of coupling variables that appears in block $k$. 
The full objective can thus be expressed as
\begin{equation}
f(x) = \sum_{k=1}^{K} f_k(s_k,x_k),
\label{eq:block_decomposition}
\end{equation}
where coupling arises through the shared components of the coupling-variable 
vector $s$, while each $x_k$ appears only in a single block.

Figure~\ref{fig:two_level_framework} illustrates the proposed two-level
framework using a simple three-block example. In this example, the objective
takes the form
\begin{equation}
f(x) =
f_1(x_1,x_2,s_1)
+ f_2(s_1,x_3,s_2)
+ f_3(s_2,x_4,x_5,x_6),
\label{eq:three_block_example}
\end{equation}
where $s_1$ couples the first and second blocks, while $s_2$ couples the
second and third blocks. The remaining variables $x_1,x_2,x_3,x_4,x_5$,
and $x_6$ are private variables associated with their respective blocks.
Thus, branching only on the coupling variables reduces the branching
dimension from eight to two. The proposed algorithm proceeds according
to the following four key components:

\begin{figure}[t] \centering \includegraphics[width=0.95\textwidth]{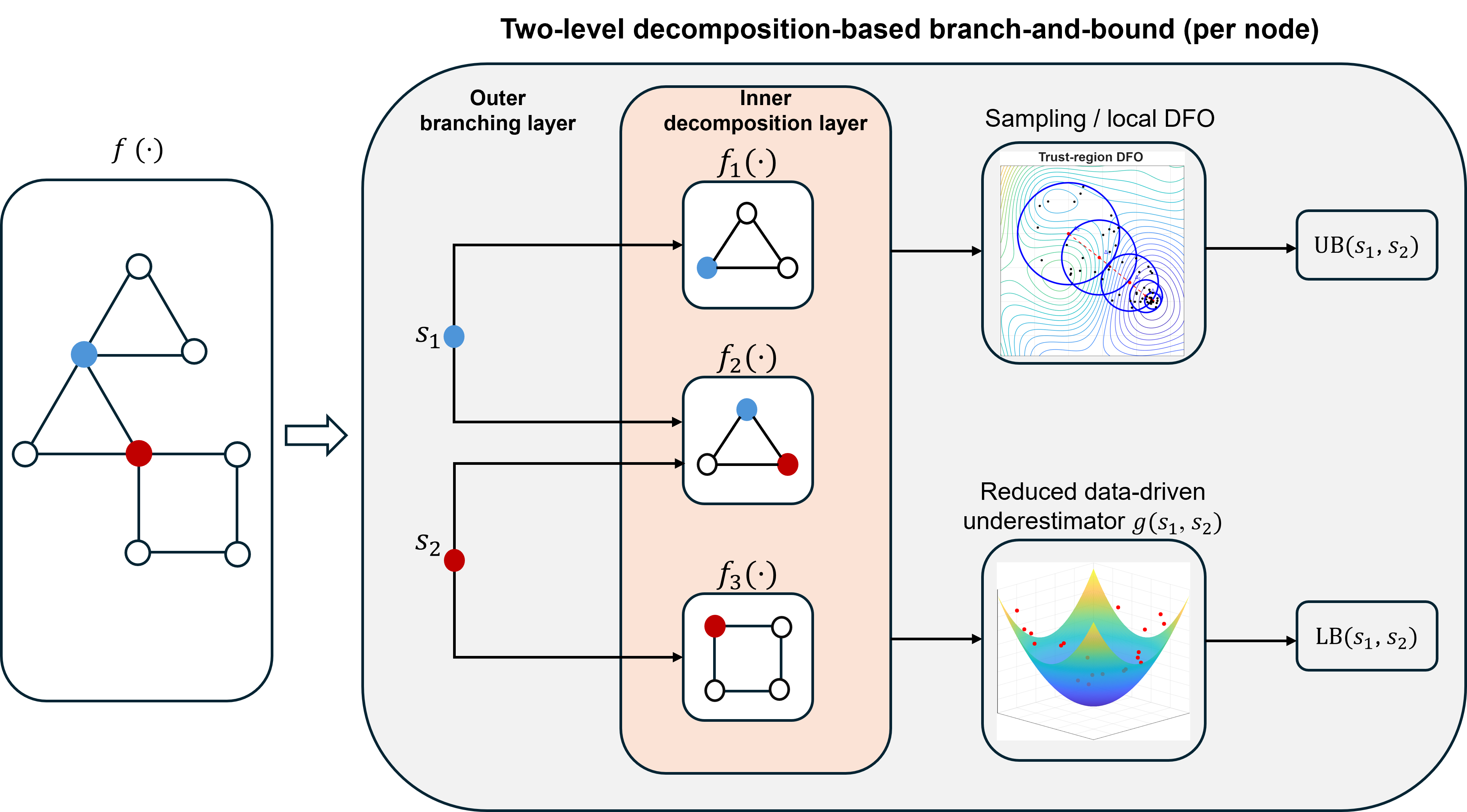} \caption{Illustration of the proposed two-level branch-and-bound framework on a three-block partially separable problem. The outer layer branches on coupling variables, while private variables are handled in the inner decomposition layer.} \label{fig:two_level_framework} \end{figure}

\begin{enumerate}[label=\arabic*.]

\item \textbf{Outer branching on coupling variables.}
The proposed framework branches only on the coupling variables $s$. Each
outer node $q$ corresponds to a coupling-variable region $S^{(q)}$, which is
progressively subdivided during the global search.

\item \textbf{Inner decomposition over private variables.}
Within each outer node, the private-variable domains are partitioned into
smaller regions, with the partition resolution progressively increased as
the algorithm proceeds. Blockwise convex underestimators are constructed
from sampled data independently across blocks and regions, enabling parallel
computation.

\item \textbf{Lower- and upper-bound construction.}
The blockwise underestimators are reduced and combined to obtain lower-bounding
models over the coupling-variable space. Node lower bounds are obtained from
these models, while upper bounds are obtained from sampled feasible points.
The latter may also be refined using trust-region derivative-free optimization
methods~\cite{conn2009,fasiku2025}.

\item \textbf{Branching, pruning, and termination.}
The resulting bounds determine whether an outer node is pruned or further
partitioned. The procedure continues in the coupling-variable space until the
prescribed termination criterion is satisfied.

\end{enumerate}

The overall procedure is summarized in Algorithm~\ref{alg:two_level_ddbb},
with the individual steps detailed in the subsequent subsections.

\begin{algorithm}[!ht]
\caption{Two-Level Data-Driven Branch-and-Bound}
\label{alg:two_level_ddbb}

\KwIn{
Coupling-variable domain $S$,
private-variable domains $\{X_k\}_{k=1}^{K}$,
objective components $\{f_k\}_{k=1}^{K}$,
initial partition resolution $M_0$,
refinement interval $L_M$,
refinement factor $\gamma>1$,
lower-bounding formulation,
and tolerance $\epsilon$
}

\KwOut{Best feasible solution $x^\star$}

Initialize the active node set $\mathcal{L}$ with the root coupling region
$S$\;

Assign the root node the lower bound $-\infty$\;

Initialize $\mathrm{UB}^{\star}\leftarrow\infty$,
$\mathrm{LB}^{\star}\leftarrow-\infty$, and
$x^\star\leftarrow\emptyset$\;

\While{$\mathcal{L}\neq\emptyset$ and
$\mathrm{UB}^{\star}-\mathrm{LB}^{\star}>\epsilon$}{

    \tcp{\textbf{Outer level: branching in the coupling-variable space}}

    Select an outer node $q\in\mathcal{L}$ with coupling region
    $S^{(q)}$ and level $\nu$, and remove it from $\mathcal{L}$\;

    Set the private-variable partition resolution as
    \[
    M_\nu=
    \left\lceil
    M_0\gamma^{\left\lfloor(\nu-1)/L_M\right\rfloor}
    \right\rceil;
    \]

    \tcp{\textbf{Inner level: decomposition over private-variable regions}}

    For each block $k=1,\ldots,K$, partition the private-variable domain
    $X_k$ using $M_\nu$ intervals per component, yielding private regions
    $X_{k,r}$, $r=1,\ldots,R_k$\;

    Generate the block-partition sample sets in
    $S_k^{(q)}\times X_{k,r}$ for all block-region pairs $(k,r)$\;

    Construct partitionwise lower-bounding models
    $f^{\mathrm{lb}}_{k,r}$ for all $(k,r)$\;

    Compute the node lower and upper bounds,
    $\mathrm{LB}^{(q)}$ and $\mathrm{UB}^{(q)}$, using
    \eqref{eq:partition_combination_form} and
    \eqref{eq:node_upper_bound}, respectively\;

    \tcp{\textbf{Outer level: node update, pruning, and branching}}

    \If{$\mathrm{UB}^{(q)}<\mathrm{UB}^{\star}$}{
        $\mathrm{UB}^{\star}\leftarrow\mathrm{UB}^{(q)}$\;
        Update $x^\star$ using the corresponding feasible point\;
    }

    \eIf{$\mathrm{LB}^{(q)}\geq\mathrm{UB}^{\star}$}{
        Prune node $q$\;
    }{
        Select the longest coupling-variable interval in $S^{(q)}$\;
        Bisect the selected interval to create two child nodes\;
        Assign $\mathrm{LB}^{(q)}$ to both child nodes\;
        Add the child nodes to $\mathcal{L}$\;
    }

    \If{$\mathcal{L}\neq\emptyset$}{
        $\displaystyle
        \mathrm{LB}^{\star}
        \leftarrow
        \min_{q\in\mathcal{L}}\mathrm{LB}^{(q)}$\;
    }
}

\Return{$x^\star$, $\mathrm{UB}^{\star}$, and $\mathrm{LB}^{\star}$}\;

\end{algorithm}

\vspace{2mm}

\subsection{Lower Bounding Procedure}
\label{sec:lower_bounding}

For a given outer node $q$, let $S^{(q)}$ denote the corresponding
coupling-variable region. As illustrated in
Figure~\ref{fig:tree_decomposition}(a), the coupling-variable domain is
partitioned through the outer branch-and-bound tree. Let $\nu$ denote the
level of node $q$, with the root node at $\nu=1$. Within each outer node, the private-variable domains are decomposed
independently across blocks, as illustrated in
Figure~\ref{fig:tree_decomposition}(b). For each block $k$, let
$S_k^{(q)}$ denote the projection of $S^{(q)}$ onto the coupling variables
appearing in block $k$, and let $X_k$ denote its private-variable domain.
The lower-bounding procedure therefore operates over the block domain
$S_k^{(q)}\times X_k$. Let $d_k$ denote the dimension of the private-variable vector $x_k$, and let
$d_{s_k}$ denote the number of coupling variables appearing in block $k$.
At tree level $\nu$, the domain of each component of $x_k$ is divided into
$M_\nu$ intervals, yielding $R_{k,\nu}=M_\nu^{d_k}$ private partitions for
block $k$. Accordingly,
\begin{equation}
X_k
=
\bigcup_{r=1}^{R_{k,\nu}} X_{k,r}^{(\nu)},
\label{eq:private_partition_domain}
\end{equation}
where $r$ indexes the private partitions. Thus, each block domain
$S_k^{(q)}\times X_k$ is decomposed into the block-region domains
$S_k^{(q)}\times X_{k,r}^{(\nu)}$, over which the partitionwise
lower-bounding models are constructed. For notational simplicity, the dependence on the tree level $\nu$ is omitted in what follows, and we write $X_{k,r}$ and $R_k$ in place of
$X_{k,r}^{(\nu)}$ and $R_{k,\nu}$, respectively.

\begin{figure}[h]
\centering
\includegraphics[width=0.98\textwidth]{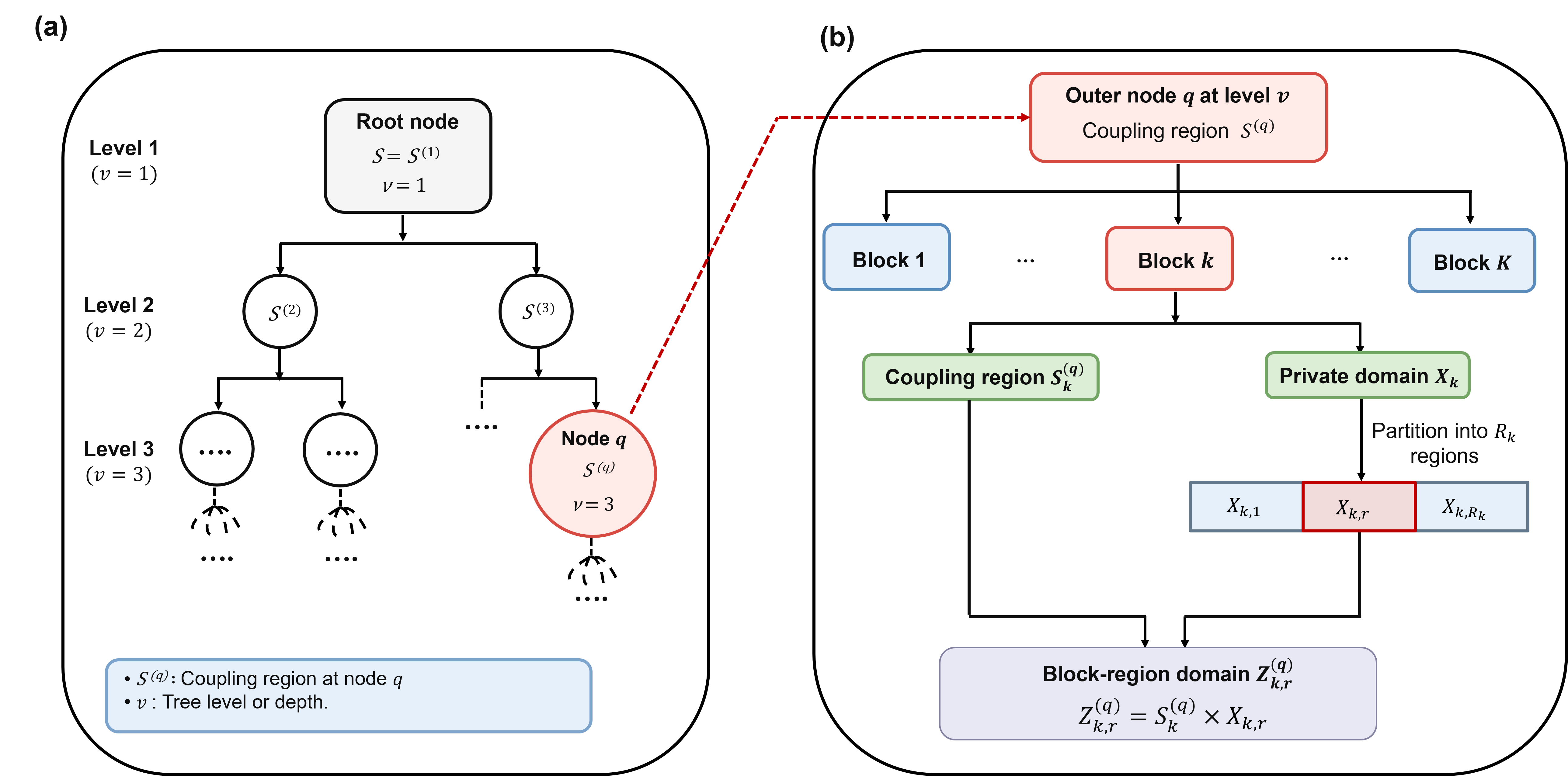}
\caption{Illustration of the two-level decomposition:
(a) outer branch-and-bound tree over the coupling variables, with a selected
node $q$ at tree level $\nu$; and
(b) inner decomposition of the private variables within the selected node.}
\label{fig:tree_decomposition}
\end{figure}

Within each partition $X_{k,r}$, the sampled coupling points are crossed with
a deterministic private-variable grid. Let $\mathcal X_{k,r}\subset X_{k,r}$
denote this grid and let
$\mathcal S_{k,r}^{\mathrm{fit}}\subset S_k^{(q)}$ denote the coupling points
used for model construction. The resulting fitting set is
\begin{equation}
\mathcal Z_{k,r}^{\mathrm{fit}}
=
\mathcal S_{k,r}^{\mathrm{fit}}
\times
\mathcal X_{k,r}
\subset
S_k^{(q)}\times X_{k,r}.
\label{eq:partition_fitting_set}
\end{equation}
Let $I_{k,r}^{\mathrm{fit}}$ denote the corresponding sample index set, with
$z^i=(s_k^i,x_k^i)\in\mathcal Z_{k,r}^{\mathrm{fit}}$. Using these samples,
we construct a partitionwise separable convex quadratic model
$m_{k,r}(s_k,x_k)$ following the data-driven convex underestimation procedure
described in Section~\ref{sec:ddcu}, based on the linear program
in~\eqref{eq:ddcu_lp}:
\begin{equation}
\begin{aligned}
m_{k,r}(s_k,x_k)
=&
\sum_{j=1}^{d_k}
a^x_{k,r,j}x_{k,j}^2
+
\sum_{j=1}^{d_k}
b^x_{k,r,j}x_{k,j}
\\
&+
\sum_{\ell=1}^{d_{s_k}}
a^s_{k,r,\ell}s_{k,\ell}^2
+
\sum_{\ell=1}^{d_{s_k}}
b^s_{k,r,\ell}s_{k,\ell}
+
c_{k,r},
\end{aligned}
\label{eq:partition_model}
\end{equation}
where
\begin{equation}
a^x_{k,r,j}\geq 0,
\qquad
a^s_{k,r,\ell}\geq 0,
\end{equation}
ensuring convexity. The absence of cross terms makes the model separable in
the private and coupling variables. The model satisfies
\begin{equation}
m_{k,r}(z^i)
\leq
f_k(z^i),
\qquad
i\in I_{k,r}^{\mathrm{fit}}.
\label{eq:sample_underestimation_condition}
\end{equation}

Although $m_{k,r}$ satisfies
\eqref{eq:sample_underestimation_condition}, the underestimation condition is
enforced only at the fitting samples and therefore does not guarantee that
$m_{k,r}$ is a valid lower bound over the entire local domain. We therefore
consider two lower-bounding formulations. The first uses the fitted convex
quadratic model directly, whereas the second incorporates a correction based
on sampled residual information to provide a lower bound beyond the fitting
samples. The two formulations employ different sampling strategies for distinct
purposes. For the first formulation, the coupling-variable fitting set consists
of samples inherited from the parent node that lie within the current coupling
region, corner points of the region, and newly generated Latin hypercube
samples. Latin hypercube sampling provides a structured space-filling design
that promotes coverage of the coupling region. However, extending this
sampling scheme to the finite-sample validity analysis developed later is less
straightforward when multiple coupling variables are present, since the
stratification of Latin hypercube sampling does not directly translate into
multidimensional cell-occupancy guarantees. Therefore, in the second
formulation, the newly generated coupling samples are instead drawn
independently and uniformly from the current coupling region. This permits a
direct probabilistic characterization of multidimensional coverage and enables
the finite-sample validity analysis developed later.

\begin{formulation}[Original lower-bounding formulation]
\label{form:original}
The partitionwise convex quadratic model is used directly as
\begin{equation}
f_{k,r}^{\mathrm{lb}}(z)
=
m_{k,r}(z),
\qquad
z=(s_k,x_k).
\label{eq:lb_formulation_one}
\end{equation}
\end{formulation}

\begin{formulation}[Residual-corrected lower-bounding formulation]
\label{form:corrected}

Define the local block-partition domain as
\begin{equation}
Z_{k,r}^{(q)}
=
S_k^{(q)}
\times
X_{k,r}.
\label{eq:local_block_partition_domain}
\end{equation}
The fill distance of the fitting set over this domain is
\begin{equation}
h_{k,r}^{(q)}
=
\sup_{z\in Z_{k,r}^{(q)}}
\min_{z^i\in\mathcal Z_{k,r}^{\mathrm{fit}}}
\|z-z^i\|.
\label{eq:local_fill_distance}
\end{equation}
Define the residual as
\begin{equation}
e_{k,r}(z)
=
f_k(z)-m_{k,r}(z).
\label{eq:residual_definition}
\end{equation}
The corrected lower-bounding model is
\begin{equation}
f_{k,r}^{\mathrm{lb}}(z)
=
m_{k,r}(z)
-
B_{k,r}h_{k,r}^{(q)},
\qquad
z\in Z_{k,r}^{(q)},
\label{eq:lb_formulation_two}
\end{equation}
where
\begin{equation}
B_{k,r}
=
\rho_{k,r}B_{k,r}^{\mathrm{raw}}
+
\omega_{k,r},
\qquad
\rho_{k,r}\geq 1,
\qquad
\omega_{k,r}>0,
\label{eq:residual_safety_factor}
\end{equation}
$\rho_{k,r}\geq 1$ is a residual-slope safety factor,
$\omega_{k,r}>0$ is an additive residual safety margin, and
$B_{k,r}^{\mathrm{raw}}$ is the raw sampled residual slope defined later.
\end{formulation}

We note that the correction $B_{k,r}h_{k,r}^{(q)}$ accounts for possible
residual variation between sampled points, since sampled underestimation
alone does not ensure validity over the full region. The validity properties
of Formulation~\ref{form:corrected} are analyzed in
Section~\ref{sec:validity_analysis}. For the remainder of the lower-bound
construction, the indices $k$, $r$, and $q$ are suppressed whenever no
ambiguity arises. Accordingly, we write $Z$, $S$, $\mathcal X$, $m$, $e$,
$h$, $B^{\mathrm{raw}}$, $\rho$, and $\omega$ for the corresponding local
quantities. The following sampling and geometric constructions are used specifically for
Formulation~\ref{form:corrected}.

\paragraph{Fitting samples.}

The coupling-variable fitting set combines
inherited, corner, and newly generated samples. Specifically,
$\mathcal S^{\mathrm{fit}}
=\mathcal S^{\mathrm{inh}}\cup\mathcal S^{\mathrm{cor}}
\cup\mathcal S^{\mathrm{iid}}$, where $\mathcal S^{\mathrm{inh}}$ contains
samples inherited from the parent node that lie within the current coupling
region, $\mathcal S^{\mathrm{cor}}$ contains the corner points of the region,
and $\mathcal S^{\mathrm{iid}}$ contains $N_s$ newly generated i.i.d.\ uniform
samples. The inherited samples reuse available function evaluations at no
additional cost, while the corner points promote coverage of the coupling
region. The newly generated samples supplement these points to provide
sufficiently rich sampling of the coupling-variable space. The complete
fitting set is then
$\mathcal Z^{\mathrm{fit}}=\mathcal S^{\mathrm{fit}}\times\mathcal X$.
Since $m$ is constructed using $\mathcal Z^{\mathrm{fit}}$ and satisfies the
sampled underestimation condition,
\begin{equation}
e(z^i)=f(z^i)-m(z^i)\geq0,
\qquad
z^i\in\mathcal Z^{\mathrm{fit}}.
\label{eq:nonnegative_fitting_residual}
\end{equation}

\paragraph{Sampled residual slopes and geometric quality.}

After $m$ has been constructed, the residual is evaluated at the points in
$\mathcal Z^{\mathrm{fit}}$. For each
$z^i\in\mathcal Z^{\mathrm{fit}}$, let $\mathcal N(i)$ denote a set of
neighboring evaluated points selected locally around $z^i$. The raw
residual-slope quantity is defined as
\begin{equation}
B^{\mathrm{raw}}
=
\max_{z^i\in\mathcal Z^{\mathrm{fit}}}
\max_{z^j\in\mathcal N(i)}
\frac{|e(z^i)-e(z^j)|}
     {\|z^i-z^j\|}.
\label{eq:raw_residual_slope}
\end{equation}

The neighbor sets are constructed from the evaluated sample locations to
provide sufficiently rich and local sampling geometry. The quality of the
resulting geometry is characterized by the directional quantity $\kappa$ and
the maximum selected-neighbor distance $d_{\mathcal N}$, whose precise
definitions are given in Section~\ref{sec:validity_analysis}. For a prescribed
constant $\Gamma_0\geq1$, we define the geometric acceptance condition
$\mathcal G$ as
\begin{equation}
\mathcal G
=
\left\{
\kappa\leq\Gamma_0,
\quad
d_{\mathcal N}\leq\Gamma_0 h
\right\}.
\label{eq:sample_geometric_quality_event}
\end{equation}
The first condition requires sufficiently rich directional coverage, while the
second requires the selected neighbors to remain sufficiently local relative
to the overall fill distance. The sample set is enriched as necessary until
the geometric acceptance condition $\mathcal G$ is satisfied. The role of
this condition in establishing validity of the residual-corrected lower bound
is developed in Section~\ref{sec:validity_analysis}.

With the two lower-bounding formulations now fully specified, we proceed to
the construction of the node lower bound. Since each
$f^{\mathrm{lb}}_{k,r}$ depends only on local data from block $k$ and
partition $r$, the lower-bounding models can be constructed independently
across block-partition pairs $(k,r)$ and therefore in parallel. In what
follows, $f^{\mathrm{lb}}_{k,r}$ denotes the model obtained from either
Formulation~\ref{form:original} or Formulation~\ref{form:corrected}. The next step is to eliminate the private variables from the lower-bounding
models. Since the outer branch-and-bound layer operates only in the
coupling-variable space, the partitionwise models
$f^{\mathrm{lb}}_{k,r}(s_k,x_k)$ are reduced to functions of the coupling
variables alone. Specifically, for each $s_k\in S_k^{(q)}$, we define
\begin{equation}
g_{k,r}(s_k)
=
\min_{x_k\in X_{k,r}}
f^{\mathrm{lb}}_{k,r}(s_k,x_k).
\label{eq:reduced_partition_function}
\end{equation}
Thus, $s_k$ is fixed when minimizing over $x_k$, and the minimization is
performed over the private-variable partition $X_{k,r}$. The separable
quadratic structure reduces this minimization to independent one-dimensional
problems of the form
\begin{equation}
\min_{x_{k,j}\in[\ell_{k,r,j},u_{k,r,j}]}
a^x_{k,r,j}x_{k,j}^2+b^x_{k,r,j}x_{k,j}.
\label{eq:one_dim_minimization}
\end{equation}

\noindent Each of these admits a closed-form solution. When $a^x_{k,r,j} > 0$, the minimizer is obtained by projecting the unconstrained optimum $-b^x_{k,r,j}/(2a^x_{k,r,j})$ onto the interval $[\ell_{k,r,j},u_{k,r,j}]$. When $a^x_{k,r,j} = 0$, the problem is linear and the minimum is attained at one of the interval endpoints. Hence, the minimization over $x_k$ can be carried out analytically with negligible computational cost.

After eliminating the private variables, the resulting function has the form
\begin{equation}
g_{k,r}(s_k)
=
\sum_{\ell=1}^{d_{s_k}} a^s_{k,r,\ell} s_{k,\ell}^2
+
\sum_{\ell=1}^{d_{s_k}} b^s_{k,r,\ell} s_{k,\ell}
+
\gamma_{k,r},
\label{eq:reduced_quadratic}
\end{equation}
where $\gamma_{k,r}$ incorporates both the intercept term and the optimal contributions from the private-variable minimizations. Since each block has multiple private partitions, the block-level lower envelope is given by
\begin{equation}
g_k(s_k)
=
\min_{r=1,\dots,R_k} g_{k,r}(s_k).
\label{eq:block_envelope}
\end{equation}

\noindent The overall reduced lower-bounding function is then
\begin{equation}
g(s)
=
\sum_{k=1}^{K}
\min_{r=1,\dots,R_k}
g_{k,r}(s_k).
\label{eq:global_reduced_function}
\end{equation}

\noindent The lower bound for node $q$ is therefore
\begin{equation}
\mathrm{LB}^{(q)}
=
\min_{s \in S^{(q)}}
\sum_{k=1}^{K}
\min_{r=1,\dots,R_k}
g_{k,r}(s_k).
\label{eq:node_lower_bound}
\end{equation}

\noindent This formulation is generally nonsmooth and nonconvex due to the minimum operators across partitions. Rather than solving this nonsmooth problem directly, we use an equivalent reformulation based on selecting one private partition per block:
\begin{equation}
\mathrm{LB}^{(q)}
=
\min_{(r_1,\dots,r_K)}
\left[
\min_{s \in S^{(q)}}
\sum_{k=1}^{K}
g_{k,r_k}(s_k)
\right],
\label{eq:partition_combination_form}
\end{equation}
where $r_k \in \{1,\dots,R_k\}$. For a fixed partition combination $(r_1,\dots,r_K)$, the inner problem is a separable convex quadratic minimization over a box. This problem is inexpensive to solve because it decomposes into independent one-dimensional minimizations over each coupling variable, each admitting a closed-form solution via projection onto its corresponding interval. The global lower bound is obtained by taking the minimum over all partition combinations. While each individual minimization is computationally trivial due to the separable quadratic structure, the number of partition combinations grows exponentially with the number of blocks and partitions. This motivates the use of pruning strategies to eliminate dominated partition combinations before enumeration, as discussed in ~\ref{app:lb_computation}.

\subsection{Upper Bounding and Node Processing}
\label{sec:node_processing}

For a given outer node $q$, an upper bound is obtained from feasible sampled 
points generated during the lower-bounding procedure within the node region. 
For each sampled realization of the coupling variables, corresponding 
private-variable samples and component-function evaluations are available 
across all blocks. Combining the blockwise samples associated with the same 
coupling-variable realization yields a consistent full-dimensional feasible 
point.

Let $\mathcal{I}^{(q)}$ denote the index set of feasible candidate solutions 
assembled within node $q$. The node upper bound is given by
\begin{equation}
\mathrm{UB}^{(q)}
=
\min_{i\in\mathcal{I}^{(q)}}
f\!\left(x^{(i)}\right),
\label{eq:node_upper_bound}
\end{equation}
where $x^{(i)}$ denotes the corresponding full-dimensional decision vector. 
Since the required component-function evaluations are already available from 
the sampling used to construct the partitionwise underestimators, this 
computation requires no additional black-box evaluations.

The global upper bound is updated using the node upper bound. The node is then 
tested for pruning. If
\begin{equation}
\mathrm{LB}^{(q)}
\ge
\mathrm{UB}^{\star},
\label{eq:node_pruning_condition}
\end{equation}
where $\mathrm{UB}^{\star}$ denotes the current global upper bound, the node is 
discarded. Otherwise, it remains active and is eligible for further exploration. For an active node, branching is performed only in the coupling-variable
space. The coupling variable with the widest current interval is selected
and bisected, producing two child nodes that partition $S^{(q)}$. This
restriction reduces the dimension of the branch-and-bound search relative
to branching in the full-dimensional space.

\section{Theoretical Analysis}
\label{sec:theoretical_analysis}

This section establishes the theoretical properties of the proposed
framework. We first analyze the validity of the residual-corrected
lower-bounding model in Formulation~\ref{form:corrected}. The convergence
properties of the overall branch-and-bound framework are then proved.

\subsection{Validity Analysis of Lower Bounding}
\label{sec:validity_analysis}

We first establish conditions under which the residual-corrected model in
Formulation~\ref{form:corrected} provides a valid lower bound throughout
the local block-partition domain $Z_{k,r}^{(q)}$. For clarity, the indices
$k$, $r$, and $q$ are suppressed whenever no ambiguity arises. Accordingly,
we write $Z$, $S$, $\mathcal X$, $m$, $e$, $h$, $B^{\mathrm{raw}}$, $\rho$,
and $\omega$ for the corresponding local quantities. All distances and
derivatives below are expressed in normalized local coordinates, with each
local domain scaled to a unit box.

\begin{assumption}[Local smoothness]
\label{assump:local_smoothness}
For each node $q$, block $k$, and partition $r$, the block function $f_k$
belongs to $C^{1,1}(Z_{k,r}^{(q)})$, i.e., $f_k$ is continuously
differentiable with Lipschitz-continuous gradient on $Z_{k,r}^{(q)}$.
There exists a constant $L_{k,r}^{(q)}>0$ such that
\begin{equation}
\|\nabla f_k(z)-\nabla f_k(y)\|
\le
L_{k,r}^{(q)}
\|z-y\|,
\qquad
z,y\in Z_{k,r}^{(q)}.
\label{eq:block_gradient_lipschitz}
\end{equation}
\end{assumption}

Since $m_{k,r}$ is quadratic,
Assumption~\ref{assump:local_smoothness} implies that the residual
$e_{k,r}=f_k-m_{k,r}$ also belongs to
$C^{1,1}(Z_{k,r}^{(q)})$, with
\begin{equation}
\|\nabla e_{k,r}(z)-\nabla e_{k,r}(y)\|
\le
H_{k,r}^{(q)}
\|z-y\|,
\qquad
z,y\in Z_{k,r}^{(q)},
\label{eq:residual_gradient_lipschitz}
\end{equation}
where
\begin{equation}
H_{k,r}^{(q)}
=
L_{k,r}^{(q)}
+
\|\nabla^2m_{k,r}\|_2.
\label{eq:residual_gradient_constant}
\end{equation}
Thus, Assumption~\ref{assump:local_smoothness} controls how rapidly the
residual gradient can vary between nearby points.

We next formalize the geometric quantities used in the acceptance condition
introduced in Section~\ref{sec:lower_bounding}. For each
$z^i\in\mathcal Z^{\mathrm{fit}}$ and
$z^j\in\mathcal N(i)$, define the unit neighbor direction
\begin{equation}
u_{ij}
=
\frac{z^j-z^i}{\|z^j-z^i\|}.
\label{eq:neighbor_direction}
\end{equation}
The directional quality of the neighbor set at $z^i$ is characterized by
\begin{equation}
\alpha_i
=
\min_{\|v\|=1}
\max_{j\in\mathcal N(i)}
|v^\top u_{ij}|,
\qquad
\kappa_i=\frac{1}{\alpha_i},
\qquad
\kappa
=
\max_{z^i\in\mathcal Z^{\mathrm{fit}}}\kappa_i.
\label{eq:directional_quality}
\end{equation}
The quantity $\alpha_i$ measures the worst-case directional coverage at
$z^i$: for every unit direction $v$, at least one selected neighbor
direction has alignment of at least $\alpha_i$. Thus, larger $\alpha_i$
corresponds to richer directional coverage, while
$\kappa_i=1/\alpha_i$ acts as a geometric amplification factor. In
particular, $0<\alpha_i\leq1$ and $\kappa_i\geq1$ whenever the selected
directions span the local joint space.

The maximum selected-neighbor distance is defined as
\begin{equation}
d_{\mathcal N}
=
\max_{z^i\in\mathcal Z^{\mathrm{fit}}}
\max_{z^j\in\mathcal N(i)}
\|z^j-z^i\|.
\label{eq:max_neighbor_distance}
\end{equation}
Hence, $\kappa$ characterizes the directional richness of the sample
geometry, whereas $d_{\mathcal N}$ characterizes the locality of the
finite-difference information. Both quantities are determined entirely by
the realized sample geometry. These quantities allow the sampled residual slopes introduced in
Section~\ref{sec:lower_bounding} to be related to the behavior of the
residual throughout the local domain.

\begin{lemma}[Residual bound from sampled slopes]
\label{lem:sampled_residual_bound}
Suppose Assumption~\ref{assump:local_smoothness} holds and the selected
neighbor directions span the local joint space. Then, for every
$z^i\in\mathcal Z^{\mathrm{fit}}$,
\begin{equation}
\|\nabla e(z^i)\|
\le
\kappa
\left(
B^{\mathrm{raw}}
+
\frac{H}{2}d_{\mathcal N}
\right).
\label{eq:sample_gradient_bound}
\end{equation}
Furthermore, using the fill distance $h$ and the nonnegative fitting
residuals in \eqref{eq:nonnegative_fitting_residual}, for every $z\in Z$,
\begin{equation}
e(z)
\ge
-
\left[
\kappa B^{\mathrm{raw}}
+
\frac{H}{2}
\left(
\kappa d_{\mathcal N}+h
\right)
\right]h.
\label{eq:global_residual_bound}
\end{equation}
\end{lemma}

The proof of Lemma~\ref{lem:sampled_residual_bound} is provided in
Appendix~\ref{app:formulation_two_validity}. Lemma~\ref{lem:sampled_residual_bound} provides a bound on the possible residual variation between the evaluated fitting points. This leads directly
to a sufficient condition for deterministic validity of the corrected model.

\begin{proposition}[Deterministic validity of Formulation~2]
\label{prop:formulation_two_deterministic}
Under the conditions of Lemma~\ref{lem:sampled_residual_bound},
the corrected model in Formulation~\ref{form:corrected} is a valid lower
bound on $Z$ whenever
\begin{equation}
\rho B^{\mathrm{raw}}+\omega
\ge
\kappa B^{\mathrm{raw}}
+
\frac{H}{2}
\left(
\kappa d_{\mathcal N}+h
\right),
\label{eq:deterministic_formulation_two_condition}
\end{equation}
or equivalently,
\begin{equation}
(\rho-\kappa)B^{\mathrm{raw}}+\omega
\ge
\frac{H}{2}
\left(
\kappa d_{\mathcal N}+h
\right).
\label{eq:deterministic_formulation_two_condition_equivalent}
\end{equation}
Whenever this condition holds,
\begin{equation}
f^{\mathrm{lb}}(z)
\le
f(z),
\qquad
z\in Z.
\label{eq:deterministic_formulation_two_validity}
\end{equation}
\end{proposition}

\begin{proof}
By Lemma~\ref{lem:sampled_residual_bound},
\begin{equation}
e(z)
\ge
-
\left[
\kappa B^{\mathrm{raw}}
+
\frac{H}{2}
\left(
\kappa d_{\mathcal N}+h
\right)
\right]h.
\end{equation}
Under \eqref{eq:deterministic_formulation_two_condition},
\begin{equation}
e(z)
\ge
-
\left(
\rho B^{\mathrm{raw}}+\omega
\right)h.
\end{equation}
Since $e(z)=f(z)-m(z)$,
\begin{equation}
m(z)
-
\left(
\rho B^{\mathrm{raw}}+\omega
\right)h
\le
f(z),
\end{equation}
which is the corrected lower-bounding model in
Formulation~\ref{form:corrected}.
\end{proof}

We now relate this deterministic condition to the geometric acceptance
condition $\mathcal G$ introduced in Section~\ref{sec:lower_bounding}.
Recall that accepted sample geometry satisfies
$\kappa\leq\Gamma_0$ and
$d_{\mathcal N}\leq\Gamma_0 h$.
The residual safety factor is chosen such that
$\rho\geq\Gamma_0$. Since $\kappa\leq\Gamma_0$ under $\mathcal G$, this
choice ensures $\rho-\kappa\geq0$ in
\eqref{eq:deterministic_formulation_two_condition_equivalent}. The
least-conservative choice satisfying this requirement is
$\rho=\Gamma_0$, leaving the remaining finite-spacing term to be controlled
through $\omega$ and the fill distance.

\begin{lemma}[Sufficient fill-distance condition]
\label{lem:sufficient_fill_distance}
Suppose Assumption~\ref{assump:local_smoothness} holds,
$\mathcal G$ occurs, and $\rho\geq\Gamma_0$. For any fixed
$\omega>0$, if
\begin{equation}
h
\leq
h_\omega
:=
\frac{2\omega}{H(\Gamma_0^2+1)},
\label{eq:h_omega_condition}
\end{equation}
then the corrected model in Formulation~\ref{form:corrected} is a valid
lower bound on $Z$.
\end{lemma}

The proof of Lemma~\ref{lem:sufficient_fill_distance} is provided in
Appendix~\ref{app:sufficient_fill_distance}. The parameter $\omega>0$ is
selected independently as part of the Formulation~2 correction. For a fixed
$\omega$, $h_\omega$ specifies a sufficient fill-distance threshold for
validity under the accepted geometry. A smaller $\omega$ requires finer
sampling, whereas a larger $\omega$ permits a larger threshold at the
expense of a more conservative correction. We next quantify the likelihood of attaining the required fill-distance
threshold under the i.i.d.\ coupling-space sampling used in
Formulation~\ref{form:corrected}.

\paragraph{Coverage under the product sampling construction.}

Let $h_X$ and $h_S$ denote the fill distances of the private-variable grid
$\mathcal X$ and coupling fitting set $\mathcal S^{\mathrm{fit}}$,
respectively. Since
$\mathcal Z^{\mathrm{fit}}=\mathcal S^{\mathrm{fit}}\times\mathcal X$
and the Euclidean distance separates across the two coordinate spaces, the
product fill distance satisfies
\begin{equation}
h^2=h_X^2+h_S^2.
\label{eq:product_fill_distance}
\end{equation}
For $h_\omega>h_X$, define
\begin{equation}
J_\omega
=
\left\lceil
\frac{\sqrt{d_{s_k}}}
{\sqrt{h_\omega^2-h_X^2}}
\right\rceil.
\label{eq:J_omega}
\end{equation}
Partition the normalized coupling box into $J_\omega^{d_{s_k}}$ equal
hypercubes. If each hypercube contains at least one fresh i.i.d.\ fitting
sample, then $h_S\leq \sqrt{d_{s_k}}/J_\omega
\leq \sqrt{h_\omega^2-h_X^2}$, and hence $h\leq h_\omega$.
A union bound therefore gives
\begin{equation}
\Pr\!\left(h>h_\omega\right)
\leq
J_\omega^{d_{s_k}}
\left(
1-\frac{1}{J_\omega^{d_{s_k}}}
\right)^{N_s}.
\label{eq:coverage_failure_probability}
\end{equation}
Inherited and corner coupling samples can only improve the realized coverage
and are therefore conservatively omitted from this probability bound. Combining this coverage result with
Lemma~\ref{lem:sufficient_fill_distance} yields a finite-sample validity
guarantee for Formulation~\ref{form:corrected}.

\begin{theorem}[Probabilistic validity of Formulation~2]
\label{thm:probabilistic_formulation_two}
Suppose Assumption~\ref{assump:local_smoothness} holds. Let the $N_s$ newly
generated coupling samples be drawn independently and uniformly from the
current coupling region and included in the fitting set. Suppose further that
\begin{equation}
\rho\geq\Gamma_0,
\qquad
\omega>0,
\qquad
h_\omega>h_X.
\end{equation}
Then
\begin{equation}
\Pr
\left(
\mathcal G
\cap
\left\{
\exists z\in Z:
f^{\mathrm{lb}}(z)>f(z)
\right\}
\right)
\leq
\delta(\omega,N_s),
\label{eq:probabilistic_formulation_two_failure}
\end{equation}
where
\begin{equation}
\delta(\omega,N_s)
=
\min
\left\{
1,\,
J_\omega^{d_{s_k}}
\left(
1-\frac{1}{J_\omega^{d_{s_k}}}
\right)^{N_s}
\right\}.
\label{eq:formulation_two_failure_probability}
\end{equation}
Thus, the probability that the geometric acceptance condition is satisfied
and the corrected lower bound is nevertheless invalid is at most
$\delta(\omega,N_s)$. Consequently, when lower bounds are used only after
the geometric acceptance condition is satisfied, the probability of an
accepted-but-invalid lower bound is bounded above by
$\delta(\omega,N_s)$.
\end{theorem}

The proof of Theorem~\ref{thm:probabilistic_formulation_two} is provided
in Appendix~\ref{app:probabilistic_formulation_two}. For fixed
$\omega>0$ with $h_\omega>h_X$, the coverage-failure bound tends to zero
as $N_s$ increases. Thus, increasing the number of fresh i.i.d.\ coupling
samples decreases the upper bound on the probability of an
accepted-but-invalid lower bound, while the geometric acceptance condition
is verified directly from the realized sample geometry.

\subsection{Convergence Analysis}
\label{sec:convergence_analysis}

We now analyze the convergence of the proposed two-level branch-and-bound
algorithm conditional on the validity of the residual-corrected lower-bounding
models in Formulation~\ref{form:corrected}. Since the original search domain
$\mathcal X=[l,u]$ is compact and Assumption~\ref{assump:local_smoothness}
implies continuity of the block functions, the objective function is
continuous on $\mathcal X$ and therefore admits at least one global minimizer. We first show that validity of the partitionwise lower-bounding models is
preserved when the private variables are eliminated and the blockwise models
are combined to form the node lower bound.

\begin{proposition}[Validity of the reduced lower bound]
\label{prop:reduced_lb_validity}
Suppose that, for every block $k$ and private partition $r$ at node $q$,
the corrected model satisfies
\begin{equation}
f^{\mathrm{lb}}_{k,r}(z)
\leq
f_k(z),
\qquad
z\in Z_{k,r}^{(q)}.
\label{eq:block_lb}
\end{equation}
Then
\begin{equation}
\mathrm{LB}^{(q)}
=
\min_{s\in S^{(q)}}
\sum_{k=1}^{K}
\min_{r=1,\ldots,R_k}
g_{k,r}(s_k)
\label{eq:node_lb_valid}
\end{equation}
is a valid lower bound on the objective over node $q$.
\end{proposition}

\begin{proof}
Fix $s\in S^{(q)}$. For every block $k$ and partition $r$,
minimizing \eqref{eq:block_lb} over $x_k\in X_{k,r}$ gives
\begin{equation}
g_{k,r}(s_k)
\leq
\min_{x_k\in X_{k,r}}
f_k(s_k,x_k).
\end{equation}
Taking the minimum over all private partitions and using
$X_k=\bigcup_{r=1}^{R_k}X_{k,r}$ yields
\begin{equation}
\min_{r=1,\ldots,R_k}
g_{k,r}(s_k)
\leq
\min_{x_k\in X_k}
f_k(s_k,x_k).
\end{equation}
Since the private variables are independent across blocks for fixed $s$,
\begin{equation}
\sum_{k=1}^{K}
\min_{r=1,\ldots,R_k}
g_{k,r}(s_k)
\leq
\min_{x_1\in X_1,\ldots,x_K\in X_K}
\sum_{k=1}^{K}
f_k(s_k,x_k).
\end{equation}
Minimizing both sides over $s\in S^{(q)}$ establishes the result.
\end{proof}

For Formulation~\ref{form:corrected}, the condition of
Proposition~\ref{prop:reduced_lb_validity} is satisfied whenever the
deterministic validity condition of
Proposition~\ref{prop:formulation_two_deterministic} holds for every
block-partition pair at the node. Thus, conditional on partitionwise validity,
the reduction and aggregation steps used to construct
$\mathrm{LB}^{(q)}$ preserve the lower-bounding property. We next characterize the refinement of the coupling and private-variable
domains.

\begin{assumption}[Exhaustive outer branching]
\label{assump:exhaustive_branching}
Along any infinite nested sequence of active outer nodes, the branching rule
satisfies
\begin{equation}
\operatorname{diam}\!\left(S^{(q)}\right)
\rightarrow 0.
\label{eq:coupling_diameter_decay}
\end{equation}
\end{assumption}

Under the longest-interval bisection rule used in the proposed algorithm,
Assumption~\ref{assump:exhaustive_branching} holds because every
coupling-variable interval is repeatedly refined along an infinite branch.

\begin{proposition}[Local domain refinement]
\label{prop:local_domain_refinement}
Under the private-partition refinement rule for $M_\nu$ in
Algorithm~\ref{alg:two_level_ddbb}, along any infinite sequence of outer
nodes with levels $\nu\rightarrow\infty$,
\begin{equation}
\max_{k,r}
\operatorname{diam}\!\left(X_{k,r}^{(\nu)}\right)
\rightarrow 0.
\label{eq:private_diameter_decay}
\end{equation}
Consequently, under Assumption~\ref{assump:exhaustive_branching},
\begin{equation}
\max_{k,r}
\operatorname{diam}\!\left(Z_{k,r}^{(q)}\right)
\rightarrow 0.
\label{eq:local_domain_diameter_decay}
\end{equation}
\end{proposition}

\begin{proof}
Since $\gamma>1$, the prescribed refinement schedule satisfies
\begin{equation}
M_\nu\rightarrow\infty
\qquad\text{as}\qquad
\nu\rightarrow\infty.
\end{equation}
Uniform partitioning of each private-variable component therefore implies
\eqref{eq:private_diameter_decay}. Moreover,
\begin{equation}
Z_{k,r}^{(q)}
=
S_k^{(q)}
\times
X_{k,r}^{(\nu)}.
\end{equation}
Assumption~\ref{assump:exhaustive_branching} gives
$\operatorname{diam}(S_k^{(q)})\rightarrow0$, while
\eqref{eq:private_diameter_decay} gives
$\operatorname{diam}(X_{k,r}^{(\nu)})\rightarrow0$.
Hence, \eqref{eq:local_domain_diameter_decay} follows.
\end{proof}

The preceding result establishes shrinking of the local domains along an
infinitely refined branch. To ensure that regions relevant to the global
lower bound are not indefinitely excluded from refinement, we impose the
following standard exhaustiveness condition on node selection.

\begin{assumption}[Exhaustive node refinement]
\label{assump:exhaustive_node_refinement}
If the algorithm does not terminate, every active region that can continue
to determine the global lower bound is eventually selected and refined.
Moreover, any nested sequence of such refined nodes that remains unpruned
reaches arbitrarily large tree levels.
\end{assumption}

Assumption~\ref{assump:exhaustive_node_refinement} prevents an active region
with a persistently low lower bound from being indefinitely excluded by the
node-selection rule. In addition to validity, convergence requires the lower-bounding models to
become tight as the corresponding local domains are refined.

\begin{assumption}[Consistency of the partitionwise lower bounds]
\label{assump:lb_consistency}
Along any infinite sequence of refined nodes,
\begin{equation}
\max_{k,r}
\sup_{z\in Z_{k,r}^{(q)}}
\left[
f_k(z)
-
f^{\mathrm{lb}}_{k,r}(z)
\right]
\rightarrow 0.
\label{eq:lb_consistency}
\end{equation}
\end{assumption}

Assumption~\ref{assump:lb_consistency} requires the partitionwise corrected
models to become asymptotically tight as the corresponding local domains are
refined. This condition is imposed separately since
$h_{k,r}^{(q)}\rightarrow0$ alone does not control the growth of
$B_{k,r}^{\mathrm{raw},(q)}$. For Formulation~\ref{form:corrected}, this
requires both the approximation error of $m_{k,r}$ and the residual-correction
term to vanish under refinement. In particular, the correction term vanishes
whenever
\begin{equation}
\left(
\rho_{k,r}^{(q)}
B_{k,r}^{\mathrm{raw},(q)}
+
\omega_{k,r}^{(q)}
\right)
h_{k,r}^{(q)}
\rightarrow0.
\label{eq:correction_decay}
\end{equation}

We can now state the global convergence result.

\begin{theorem}[Global convergence]
\label{thm:convergence}
Suppose Assumptions~\ref{assump:exhaustive_branching},
\ref{assump:exhaustive_node_refinement}, and
\ref{assump:lb_consistency} hold. Suppose further that every accepted
partitionwise corrected model satisfies the deterministic validity condition
of Proposition~\ref{prop:formulation_two_deterministic}. Then the
residual-corrected two-level branch-and-bound algorithm is globally
convergent. In particular, if the algorithm terminates with
\begin{equation}
\mathrm{UB}^{\star}
-
\mathrm{LB}^{\star}
\leq
\epsilon,
\label{eq:global_termination}
\end{equation}
where $\mathrm{LB}^{\star}$ denotes the global lower bound over the active
nodes, then the incumbent solution $x^\star$ satisfies
\begin{equation}
f(x^\star)-f^\star
\leq
\epsilon,
\label{eq:epsilon_global_optimality}
\end{equation}
where $f^\star$ denotes the global optimal objective value.
\end{theorem}

\begin{proof}
By Proposition~\ref{prop:formulation_two_deterministic}, every accepted
partitionwise corrected model is a valid lower bound on its corresponding
local domain. Proposition~\ref{prop:reduced_lb_validity} therefore implies
that every node lower bound is valid. Since the incumbent is obtained from a
feasible sampled point,
\begin{equation}
\mathrm{LB}^{\star}
\leq
f^\star
\leq
\mathrm{UB}^{\star}
=
f(x^\star).
\label{eq:global_bound_relation}
\end{equation}
Consequently, the pruning condition
$\mathrm{LB}^{(q)}\geq\mathrm{UB}^{\star}$ cannot eliminate a node containing
a point that strictly improves the current incumbent.

Consider an infinite execution of the algorithm. If a node containing a
global minimizer is pruned, validity of its lower bound implies that the
current incumbent already attains the global optimal value. Otherwise, at
least one branch containing a global minimizer remains active. By
Assumption~\ref{assump:exhaustive_node_refinement}, any active regions that
continue to determine the global lower bound are eventually refined to
arbitrarily large tree levels. Assumption~\ref{assump:exhaustive_branching}
and Proposition~\ref{prop:local_domain_refinement} then imply that the
corresponding coupling and private-variable regions shrink to zero diameter.
By Assumption~\ref{assump:lb_consistency}, the associated lower bounds become
asymptotically tight with respect to the objective over these shrinking
regions. Since the node lower bounds remain valid, it follows that
\begin{equation}
\mathrm{LB}^{\star}\rightarrow f^\star.
\end{equation}

Along a nested sequence of refined regions containing a global minimizer, the
sampling construction provides feasible sampled points within the
corresponding regions. Since the diameters of these regions converge to zero,
a sequence of such feasible sampled points converges to a global minimizer.
By continuity of $f$, their objective values converge to $f^\star$. Since the
incumbent is the best feasible objective value obtained by the algorithm,
\begin{equation}
\mathrm{UB}^{\star}\rightarrow f^\star.
\end{equation}
Therefore,
\begin{equation}
\mathrm{UB}^{\star}
-
\mathrm{LB}^{\star}
\rightarrow0,
\end{equation}
which establishes global convergence. Finally, if the algorithm terminates according to
\eqref{eq:global_termination}, then
\begin{equation}
f(x^\star)-f^\star
=
\mathrm{UB}^{\star}-f^\star
\leq
\mathrm{UB}^{\star}-\mathrm{LB}^{\star}
\leq
\epsilon.
\label{eq:epsilon_optimality_proof}
\end{equation}
Therefore, $x^\star$ is an $\epsilon$-global optimizer.
\end{proof}

\section{Numerical Examples}
\label{sec:numerical}

This section evaluates the proposed two-level branch-and-bound framework on four 
structured CUTEst benchmark problems and two ODE-constrained parameter-estimation 
problems. The benchmark problems are used to assess the scalability of the proposed 
framework, while the ODE examples demonstrate its applicability to parameter-estimation 
problems with decomposable structure.

Unless otherwise stated, all computations are performed on a high-performance
computing system using MPI with 100 processors. Two variants of the proposed
method are considered, corresponding to the lower-bounding constructions in
Formulations~\ref{form:original} and~\ref{form:corrected}. These are hereafter
referred to as Proposed-F1 and Proposed-F2, respectively. Both are compared
with the data-driven branch-and-bound method of Zhai et al.~\cite{zhai2022a},
hereafter referred to as the monolithic method. Performance is evaluated in terms of computational time and the number of function evaluations (\#Fev). For the proposed distributed methods, the reported
\#Fev is the maximum number of blockwise function evaluations across all blocks,
providing a conservative measure of the equivalent evaluation cost under parallel
execution.

\subsection{Benchmark Functions}
\label{sec:benchmark_examples}
We first evaluate the proposed framework on four scalable structured benchmark 
problems selected from the CUTEst test set~\cite{gould2015cutest}: ARWHEAD, 
LIARWHD, NONDIA, and EG2. These problems possess partially separable structures 
compatible with the proposed decomposition, with a small set of coupling variables 
shared across objective components and private variables associated with individual 
blocks. Detailed descriptions of the problems and their decompositions are provided 
in \ref{app:benchmark_problems}. We consider three problem-dimension regimes: low ($n=3,4$), medium
($n=10,20$), and high ($n=50$). Following the proposed 
framework, branching is restricted to the low-dimensional coupling-variable space, 
while the private variables are handled through the blockwise lower-bounding 
procedure. For all benchmark problems, the decision variables are restricted to
$-2\leq x_i\leq 2$. The initial private-variable partition resolution is set to
$M_0=20$ intervals per variable, with refinement factor $\gamma=1.2$ and
refinement frequency $L_M=15$ branch-and-bound levels. For Formulation~2, the
geometric threshold and residual safety factor are set to
$\Gamma_0=\rho=2$, and the additive residual margin is set to
$\omega=0.01$. The partitionwise underestimators and the associated lower-bound
computations are evaluated independently in parallel on a high-performance
computing system. All problems are solved using an absolute optimality tolerance
of $0.1$ and a computational time limit of $3600$~s.

Table~\ref{tab:benchmark_results} reports the computational time and number of
function evaluations required by Proposed-F1, Proposed-F2, and the monolithic
method. Across the benchmark set, both proposed methods exhibit substantially
better scalability with increasing problem dimension than the monolithic
method. This advantage becomes increasingly pronounced for the medium- and
higher-dimensional instances, reflecting the efficiency of the proposed
decomposition-based framework. Of the two proposed methods, Proposed-F1 generally
requires shorter computational times and fewer function evaluations to satisfy
the prescribed optimality tolerance. This behavior is consistent with its tighter
lower bounds, which enable more effective pruning of the outer branch-and-bound
tree. In contrast, Proposed-F2 incorporates the residual-based correction in
Formulation~\ref{form:corrected}, resulting in more conservative lower bounds
and, consequently, reduced pruning efficiency. Nevertheless, both proposed
methods retain a clear scalability advantage over the monolithic method as
the problem dimension increases.

\begin{table*}[htbp]
\centering
\caption{Computational results for the CUTEst benchmark problems.}
\label{tab:benchmark_results}

\begin{adjustbox}{max width=\textwidth}
\begin{tabular}{cccccccc}
\toprule
\multirow{2}{*}{Problem}
& \multirow{2}{*}{$n$}
& \multicolumn{2}{c}{Monolithic method}
& \multicolumn{2}{c}{Proposed-F1}
& \multicolumn{2}{c}{Proposed-F2} \\
\cmidrule(lr){3-4}
\cmidrule(lr){5-6}
\cmidrule(lr){7-8}
&
& Time (s) & \#Fev
& Time (s) & \#Fev
& Time (s) & \#Fev \\
\midrule

\multirow{5}{*}{ARWHEAD}
& 3
& 2.78 & \textbf{1{,}001}
& \textbf{0.86} & 1{,}260
& 2.76 & 5{,}019 \\

& 4
& 13.68 & 5{,}350
& \textbf{0.63} & \textbf{1{,}386}
& 2.57 & 5{,}019 \\

& 10
& $\dagger$ & $\dagger$
& \textbf{1.27} & \textbf{1{,}554}
& 7.02 & 5{,}775 \\

& 20
& $\dagger$ & $\dagger$
& \textbf{2.21} & \textbf{1{,}764}
& 12.39 & 5{,}775 \\

& 50
& $\dagger$ & $\dagger$
& \textbf{5.15} & \textbf{1{,}764}
& 32.13 & 6{,}531 \\

\midrule

\multirow{5}{*}{LIARWHD}
& 3
& 7.53 & \textbf{2{,}470}
& \textbf{2.47} & 5{,}166
& 6.90 & 17{,}955 \\

& 4
& 29.34 & 11{,}320
& \textbf{2.27} & \textbf{5{,}355}
& 7.82 & 19{,}593 \\

& 10
& $\dagger$ & $\dagger$
& \textbf{10.41} & \textbf{11{,}298}
& 38.07 & 45{,}612 \\

& 20
& $\dagger$ & $\dagger$
& \textbf{31.17} & \textbf{16{,}653}
& 109.89 & 67{,}200 \\

& 50
& $\dagger$ & $\dagger$
& \textbf{2{,}488.50} & \textbf{24{,}213}
& $\dagger$ & $\dagger$ \\

\midrule

\multirow{5}{*}{NONDIA}
& 3
& 66.30 & \textbf{24{,}763}
& \textbf{26.13} & 68{,}712
& 87.28 & 263{,}928 \\

& 4
& 205.50 & 83{,}507
& \textbf{32.54} & \textbf{80{,}682}
& 107.67 & 310{,}968 \\

& 10
& $\dagger$ & $\dagger$
& \textbf{79.37} & \textbf{92{,}862}
& 307.40 & 421{,}911 \\

& 20
& $\dagger$ & $\dagger$
& \textbf{219.25} & \textbf{133{,}812}
& 727.16 & 518{,}490 \\

& 50
& $\dagger$ & $\dagger$
& \textbf{681.42} & \textbf{167{,}412}
& 2{,}297.97 & 660{,}765 \\

\midrule

\multirow{5}{*}{EG2}
& 3
& 2.34 & \textbf{776}
& \textbf{0.90} & 1{,}407
& 2.51 & 5{,}040 \\

& 4
& 6.91 & 2{,}398
& \textbf{0.96} & \textbf{2{,}163}
& 2.80 & 6{,}258 \\

& 10
& $\dagger$ & $\dagger$
& \textbf{3.19} & \textbf{4{,}032}
& 14.14 & 16{,}569 \\

& 20
& $\dagger$ & $\dagger$
& \textbf{9.01} & \textbf{7{,}119}
& 42.89 & 27{,}363 \\

& 50
& $\dagger$ & $\dagger$
& \textbf{32.48} & \textbf{11{,}193}
& 180.15 & 50{,}610 \\

\bottomrule
\end{tabular}
\end{adjustbox}

\vspace{0.5em}
\raggedright
$\dagger$ Method did not converge within the 3600~s computational time limit.

\end{table*}

\subsection{ODE-Constrained Parameter-Estimation Problems}
\label{sec:parameter_estimation}

\subsubsection{Nonlinear Mixed-Effects (NLME) Parameter Estimation}

We consider a synthetic NLME--ODE parameter-estimation problem inspired by 
the framework of Clairon et al.~\cite{clairon2024parameter}. For subject $i$, 
the state variables satisfy
\begin{align}
\frac{dX_{1,i}}{dt}
&=
\phi_2 X_{2,i}
-
\phi_{1,i}X_{1,i},
\\
\frac{dX_{2,i}}{dt}
&=
-\phi_2 X_{2,i},
\end{align}
where
\begin{equation}
\phi_{1,i}
=
\exp(\theta_1+b_i),
\qquad
\phi_2
=
\exp(\theta_2).
\end{equation}
Here, $\theta_1$ and $\theta_2$ are shared population parameters and $b_i$ is 
the subject-specific random effect. Synthetic observations of $X_{1,i}$ are 
generated at $11$ equally spaced time points over $[0,10]$ using 
$\theta_1=\log(0.5)$, $\theta_2=\log(0.2)$, measurement-noise standard 
deviation $\sigma=0.05$, and random-effect standard deviation $\sigma_b=0.5$.

With $\sigma$ and $\sigma_b$ fixed, the parameter-estimation problem is
\begin{equation}
\min_{\theta_1,\theta_2,\{b_i\}}
\sum_{i=1}^{N_{\mathrm{sub}}}
\left[
\sum_{j=1}^{N_t}
\left(
y_{ij}
-
X_{1,i}(t_j;\theta_1,\theta_2,b_i)
\right)^2
+
\frac{\sigma^2}{\sigma_b^2}b_i^2
\right],
\label{eq:nlme_objective}
\end{equation}
where $y_{ij}$ denotes the observation for subject $i$ at time $t_j$ and
$N_t=11$ is the number of observation time points. The problem has a natural star structure, with 
$s=(\theta_1,\theta_2)$ as the coupling-variable vector and $b_i$ as the 
private variable for subject $i$. We consider $N_{\mathrm{sub}}=10$ subjects, 
resulting in a $12$-dimensional optimization problem with two coupling variables 
and ten private variables. The parameter bounds are
\begin{equation}
\log(0.1)\leq\theta_1\leq\log(2),
\qquad
\log(0.05)\leq\theta_2\leq\log(1),
\qquad
-2\leq b_i\leq2.
\end{equation}
Each private variable is initially partitioned into $30$ intervals, and the 
resulting blockwise lower-bound computations are evaluated independently in 
parallel.

Figure~\ref{fig:nlme_results} presents the results for the NLME--ODE problem.
Proposed-F1 and Proposed-F2 reduce the absolute optimality gap substantially
faster than the monolithic method, with Proposed-F1 exhibiting the fastest
convergence. The fitted trajectory obtained using Proposed-F1 after $2000$~s
also closely matches the true trajectory and noisy observations for the last
subject.

\begin{figure}[H]
\centering
\includegraphics[width=0.95\textwidth]{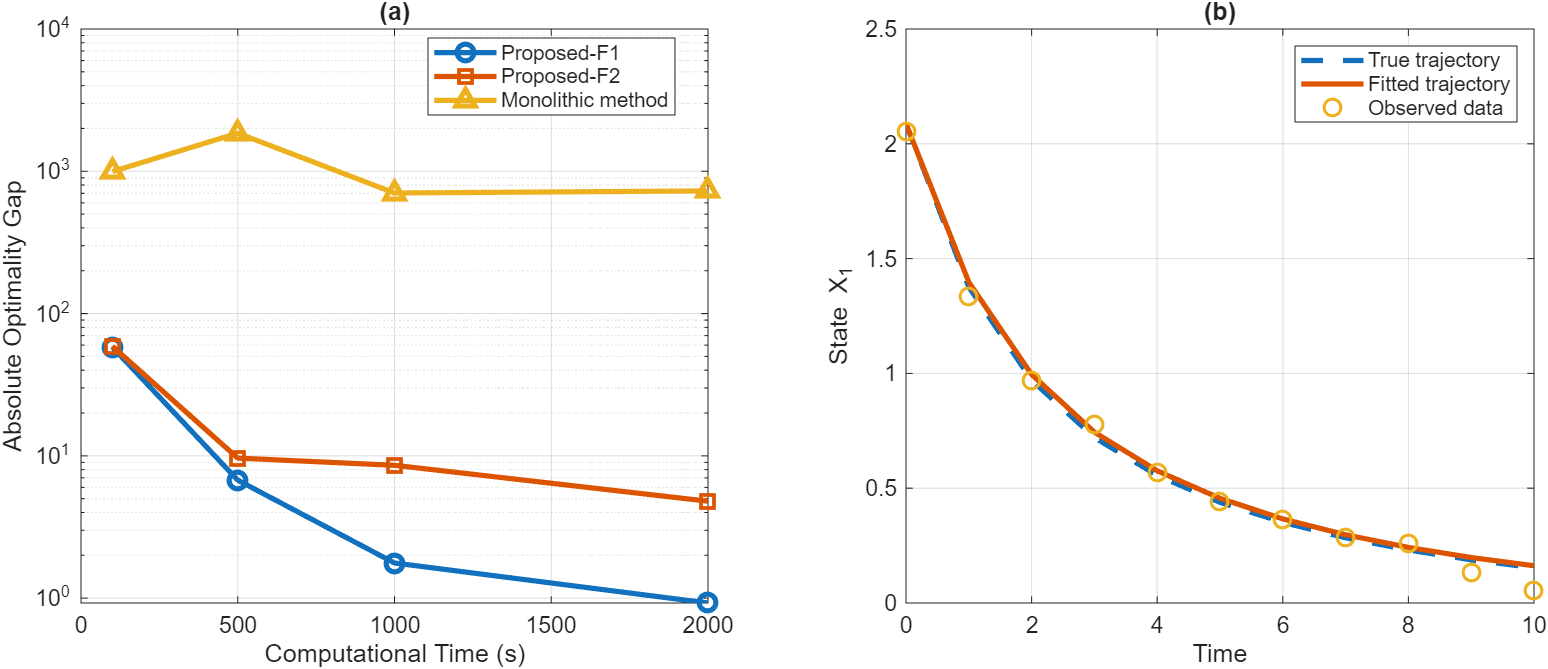}
\caption{Results for the $N_{\mathrm{sub}}=10$ NLME--ODE parameter-estimation 
problem: (a) absolute optimality gap versus computational time for Proposed-F1,
Proposed-F2, and the monolithic method; (b) noisy observations, true trajectory,
and fitted trajectory for the last subject using Proposed-F1 after $2000$~s.}
\label{fig:nlme_results}
\end{figure}

\subsubsection{ODE-Constrained Kinetic Parameter Estimation}

Finally, we consider an ODE-constrained kinetic parameter-estimation problem
adapted from the global dynamic optimization case study of
Singer~\cite{singer2004}.
The original study considers experiments conducted at 273, 298, and 323~K, 
where kinetic parameters are estimated by fitting a nonlinear ODE model to 
measured optical-density data. For a temperature $T$, let $x_A$, $x_Z$, $x_Y$, $x_D$, and $x_B$ denote the 
species concentrations. Their dynamics are governed by

\begin{equation}
\begin{aligned}
\dot{x}_A
&=
k_1x_Zx_Y
-\left[\exp(k'_{2f,T})+\exp(k'_{3f,T})\right]x_Ax_{O_2,T}
\\[-0.2em]
&\hspace{3.8cm}
+\frac{\exp(k'_{2f,T})}{K_{2,T}}x_D
+\frac{\exp(k'_{3f,T})}{K_{3,T}}x_B
-k_5x_A^2,
\\[0.3em]
\dot{x}_Z
&=
-k_1x_Zx_Y,
\\
\dot{x}_Y
&=
-k_1x_Zx_Y,
\\
\dot{x}_D
&=
\exp(k'_{2f,T})x_Ax_{O_2,T}
-\frac{\exp(k'_{2f,T})}{K_{2,T}}x_D,
\\
\dot{x}_B
&=
\exp(k'_{3f,T})x_Ax_{O_2,T}
-\left[
\frac{\exp(k'_{3f,T})}{K_{3,T}}
+\exp(k'_{4,T})
\right]x_B.
\end{aligned}
\label{eq:kinetic_ode_system}
\end{equation}
\noindent The initial conditions are
$x_Z(0)=x_{Z,0,T}$, $x_Y(0)=x_{Y,0,T}$, and
$x_A(0)=x_B(0)=x_D(0)=0$, where
$K_{2,T}=46\exp\!\left(6500/T-18\right)$ and
$K_{3,T}=2K_{2,T}$.
\noindent The model-predicted optical density is
\begin{equation}
I_{M,T,j}
=
2100x_A(t_{j,T})
+
200\left[x_B(t_{j,T})+x_D(t_{j,T})\right],
\end{equation}
and the corresponding least-squares objective for temperature $T$ is
\begin{equation}
J_T
=
\sum_{j=1}^{N_T}
\left(
I_{E,T,j}
-
I_{M,T,j}
\right)^2,
\end{equation}
where $I_{E,T,j}$ denotes the measured optical density at time $t_{j,T}$.

In this work, $k_1$ and $k_5$ are estimated as shared parameters across all 
temperature experiments. To assess scalability, the three original datasets 
are augmented with interpolated cases at 280, 290, 300, 310, and 320~K, giving
\begin{equation}
\mathcal{T}
=
\{273,280,290,298,300,310,320,323\}\ \mathrm{K}.
\end{equation}
\noindent The coupling-variable vector is
$\mathbf{s}=(\log k_1,\log k_5)$,
while each temperature block has the private-variable vector
$\mathbf{z}_T=(k'_{2f,T},k'_{3f,T})$. The parameter $k'_{4,T}$ is fixed from interpolated values of the original 
estimates. The resulting optimization problem is
\begin{equation}
\min_{\mathbf{s},\,\{\mathbf{z}_T\}_{T\in\mathcal{T}}}
\quad
\sum_{T\in\mathcal{T}} J_T,
\end{equation}
subject to the ODE system for each temperature block and the bounds
\begin{equation}
\begin{aligned}
\log(20)
&\le \log(k_1)\le \log(250), \\
\log(100)
&\le \log(k_5)\le \log(2500), \\
2.303
&\le k'_{2f,T},\,k'_{3f,T}\le 7.090,
\qquad T\in\mathcal{T}.
\end{aligned}
\label{eq:kinetic_parameter_bounds}
\end{equation}

\noindent The problem therefore contains two coupling variables and sixteen private 
variables, yielding an $18$-dimensional optimization problem. Each private 
variable is initially partitioned into $30$ intervals, and the resulting blockwise 
lower-bound computations are evaluated independently in parallel. Figure~\ref{fig:kinetic_results} presents the results for the kinetic reaction
problem. Consistent with the NLME--ODE example, Proposed-F1 and Proposed-F2
reduce the absolute optimality gap substantially faster than the monolithic
method, with Proposed-F1 exhibiting the fastest convergence. The fitted
optical-density trajectory at $273$~K also closely agrees with the measured
data.

\begin{figure}[H]
\centering
\includegraphics[width=0.95\textwidth]{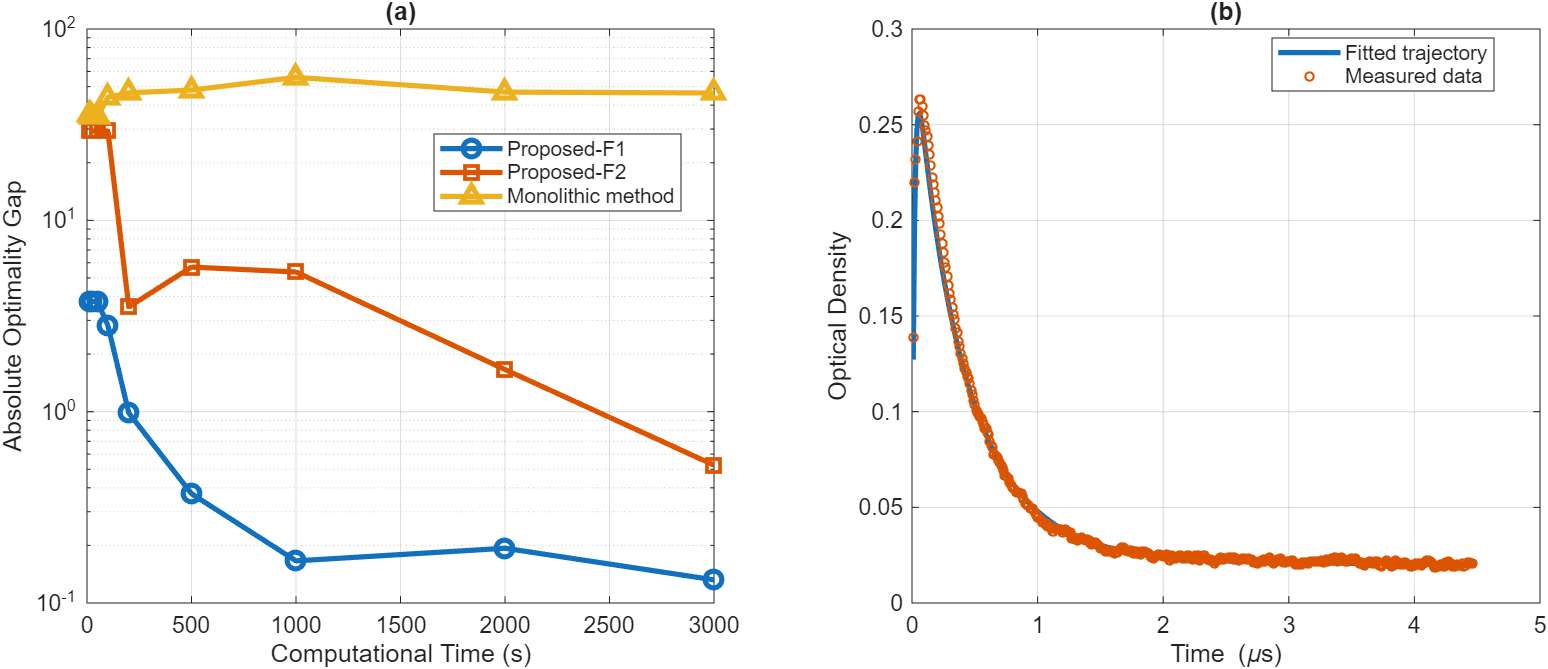}
\caption{Results for the $18$-dimensional ODE-constrained kinetic
parameter-estimation problem: (a) absolute optimality gap versus computational
time for Proposed-F1, Proposed-F2, and the monolithic method; and (b) measured
and fitted optical-density trajectories at $273$~K using Proposed-F1 after
$3000$~s.}
\label{fig:kinetic_results}
\end{figure}

\section{Conclusion}\label{sec:conclusion}

In this work, we have proposed a distributed data-driven branch-and-bound framework for the global optimization of derivative-free problems, with a particular focus on problems exhibiting exploitable structure. The method adopts a two-level architecture, where the key idea is to perform branching on a lower-dimensional set of coupling variables, while the inner-level problems are solved in parallel across decomposed regions using modern multi-core computing architectures. We have established convergence of the proposed method under suitable assumptions and demonstrated its computational advantages through numerical experiments. In particular, the results show significant improvements in computational time and, in many cases, function evaluation efficiency when compared to monolithic approaches.

The proposed framework is especially well-suited for partially separable problems in which the dimension of the coupling variables is significantly smaller than that of the original problem. In such settings, the method effectively reduces the branching complexity. Furthermore, the distributed nature of the approach can be interpreted as transferring the exploration of multiple regions, which would traditionally occur sequentially within the branch-and-bound tree, into an inner layer where they can be processed simultaneously in parallel. An important aspect of the method is the use of separable convex quadratic underestimators. This structure enables closed-form solutions for the lower-bounding problems, thereby avoiding the need for iterative solvers and contributing to the overall computational efficiency of the framework.

Finally, we emphasize that the convergence and global optimality guarantees of the proposed framework depend on the validity of the lower-bounding constructions employed within the branch-and-bound procedure. In this work, two alternative lower-bounding formulations have been considered. Formulation~1 yields tighter lower bounds and superior computational performance in the numerical experiments, but underestimation is enforced only at the sampled points and does not by itself guarantee validity over the full region. Formulation~2 addresses this issue by incorporating an additional residual term designed to improve the theoretical validity of the resulting lower bound. The numerical results indicate that this added conservativeness comes at the expense of weaker lower bounds and increased computational effort. These observations motivate the broader challenge of achieving both data-efficient exploration and reliable global optimization for high-dimensional black-box problems. 

A particularly interesting future direction is the integration of the proposed decomposition framework with Bayesian optimization, whose data-efficient sampling strategies face significant scalability challenges and generally lack finite-time global-optimality certificates. Such an integration may provide a pathway toward scalable Bayesian optimization with stronger global guarantees.
\vspace{6mm}

\appendix
\renewcommand{\thesection}{Appendix \Alph{section}}
\renewcommand{\thesubsection}{\Alph{section}.\arabic{subsection}}


\section{Efficient Evaluation of the Lower Bounding Problem}
\label{app:lb_computation}

From \eqref{eq:partition_combination_form}, the lower bound associated with
node $q$ is
\begin{equation}
\mathrm{LB}^{(q)}
=
\min_{(r_1,\dots,r_K)}
\left[
\min_{s \in S^{(q)}}
\sum_{k=1}^{K}
g_{k,r_k}(s_k)
\right],
\label{eq:appendix_lb}
\end{equation}
where $r_k\in\{1,\dots,R_k\}$ denotes the selected private partition for
block $k$. Although the private variables have been eliminated, the number
of partition combinations is $\prod_{k=1}^{K}R_k$, making direct enumeration
potentially prohibitive.

For a fixed combination $(r_1,\dots,r_K)$, the inner problem is
\begin{equation}
\min_{s\in S^{(q)}}
\sum_{k=1}^{K}g_{k,r_k}(s_k).
\label{eq:fixed_combination_problem}
\end{equation}
Since the reduced models are separable convex quadratics, this problem
decomposes into one-dimensional convex quadratic minimizations, each solved
by projecting the unconstrained minimizer onto the corresponding interval.
Thus, individual combinations can be evaluated efficiently, and the main
challenge is reducing the number of combinations considered.

First, dominated partition models are removed. For two reduced models
$g_{k,r}$ and $g_{k,r'}$ of the same block, define
\begin{equation}
\Delta_{r,r'}(s_k)
=
g_{k,r}(s_k)-g_{k,r'}(s_k).
\end{equation}
If
\begin{equation}
\min_{s_k\in S_k^{(q)}}
\Delta_{r,r'}(s_k)
\geq0,
\label{eq:dominance_condition}
\end{equation}
then $g_{k,r}(s_k)\geq g_{k,r'}(s_k)$ for all
$s_k\in S_k^{(q)}$. Hence, partition $r$ is dominated by $r'$ and can be
discarded without affecting the optimal value of \eqref{eq:appendix_lb}.

After dominance pruning, the remaining combinations are explored using an
auxiliary branch-and-bound procedure. A node corresponds to a partial
selection $(r_1,\dots,r_p)$, with
\begin{equation}
\phi_p(s)
=
\sum_{k=1}^{p}g_{k,r_k}(s_k).
\end{equation}
For each unfixed block, define
\begin{equation}
\underline g_k
=
\min_r\;
\min_{s_k\in S_k^{(q)}}
g_{k,r}(s_k).
\end{equation}
A valid lower bound for all completions of the partial selection is then
\begin{equation}
\underline{\phi}_p
=
\min_{s\in S^{(q)}}\phi_p(s)
+
\sum_{k=p+1}^{K}\underline g_k.
\label{eq:bb_bound}
\end{equation}

Let $V^\star$ denote the best value among fully evaluated partition
combinations. This quantity is an incumbent for the auxiliary combinatorial
problem and is distinct from the upper bound of the outer branch-and-bound
algorithm. If $\underline{\phi}_p\geq V^\star$, the corresponding branch can
be safely pruned. Together, dominance pruning and the auxiliary
branch-and-bound procedure substantially reduce the number of partition
combinations that must be evaluated.


\section{Proofs}
\label{app:proofs}

\subsection{Proof of Lemma~\ref{lem:sampled_residual_bound}}
\label{app:formulation_two_validity}

\begin{proof}
Fix $z^i\in\mathcal Z^{\mathrm{fit}}$ and
$j\in\mathcal N(i)$. Since the residual gradient is
$H$-Lipschitz, the $C^{1,1}$ Taylor bound gives
\begin{equation}
\left|
e(z^j)-e(z^i)
-
\nabla e(z^i)^\top(z^j-z^i)
\right|
\le
\frac{H}{2}\|z^j-z^i\|^2.
\label{eq:app_residual_taylor}
\end{equation}
Dividing by $\|z^j-z^i\|$ and using
\eqref{eq:neighbor_direction} and \eqref{eq:raw_residual_slope} yields
\begin{equation}
\left|
\nabla e(z^i)^\top u_{ij}
\right|
\le
B^{\mathrm{raw}}
+
\frac{H}{2}\|z^j-z^i\|
\le
B^{\mathrm{raw}}
+
\frac{H}{2}d_{\mathcal N}.
\label{eq:app_directional_derivative_bound}
\end{equation}

If $\nabla e(z^i)\neq0$, let
\[
v_i
=
\frac{\nabla e(z^i)}
     {\|\nabla e(z^i)\|}.
\]
By the definition of $\alpha_i$, there exists
$j\in\mathcal N(i)$ such that
\[
|v_i^\top u_{ij}|
\ge
\alpha_i.
\]
Therefore,
\begin{equation}
\alpha_i\|\nabla e(z^i)\|
\le
B^{\mathrm{raw}}
+
\frac{H}{2}d_{\mathcal N}.
\end{equation}
Since $\kappa_i=1/\alpha_i$ and $\kappa_i\le\kappa$,
\begin{equation}
\|\nabla e(z^i)\|
\le
\kappa
\left(
B^{\mathrm{raw}}
+
\frac{H}{2}d_{\mathcal N}
\right),
\end{equation}
which also holds trivially when $\nabla e(z^i)=0$. This proves
\eqref{eq:sample_gradient_bound}.

Now fix any $z\in Z$. By the definition of the fill distance, there exists
a fitting point $z^i\in\mathcal Z^{\mathrm{fit}}$ such that
\[
\|z-z^i\|\le h.
\]
Applying the $C^{1,1}$ Taylor bound at $z^i$ gives
\begin{equation}
e(z)
\ge
e(z^i)
+
\nabla e(z^i)^\top(z-z^i)
-
\frac{H}{2}\|z-z^i\|^2.
\label{eq:app_residual_lower_taylor}
\end{equation}
Since $e(z^i)\ge0$ by
\eqref{eq:nonnegative_fitting_residual},
\begin{align}
e(z)
&\ge
-
\|\nabla e(z^i)\|\,\|z-z^i\|
-
\frac{H}{2}\|z-z^i\|^2
\\
&\ge
-
\kappa
\left(
B^{\mathrm{raw}}
+
\frac{H}{2}d_{\mathcal N}
\right)h
-
\frac{H}{2}h^2.
\end{align}
This proves \eqref{eq:global_residual_bound}.
\end{proof}

\subsection{Proof of Lemma~\ref{lem:sufficient_fill_distance}}
\label{app:sufficient_fill_distance}

\begin{proof}
Under the geometric acceptance condition $\mathcal G$,
$\kappa\leq\Gamma_0$ and $d_{\mathcal N}\leq\Gamma_0 h$.
Since $\rho\geq\Gamma_0$, it follows that $\rho-\kappa\geq0$.
Because $B^{\mathrm{raw}}\geq0$,
\begin{equation}
(\rho-\kappa)B^{\mathrm{raw}}
\geq0.
\label{eq:app_nonnegative_residual_margin}
\end{equation}

By the geometric bounds,
\begin{equation}
\frac{H}{2}
\left(
\kappa d_{\mathcal N}+h
\right)
\leq
\frac{H}{2}
\left(
\Gamma_0^2+1
\right)h.
\label{eq:app_geometric_smoothness_bound}
\end{equation}

If $h\leq h_\omega=2\omega/[H(\Gamma_0^2+1)]$, then
\begin{equation}
\frac{H}{2}
\left(
\Gamma_0^2+1
\right)h
\leq
\omega.
\label{eq:app_omega_dominates_smoothness}
\end{equation}
Combining
\eqref{eq:app_nonnegative_residual_margin},
\eqref{eq:app_geometric_smoothness_bound}, and
\eqref{eq:app_omega_dominates_smoothness} gives
\begin{equation}
(\rho-\kappa)B^{\mathrm{raw}}
+
\omega
\geq
\frac{H}{2}
\left(
\kappa d_{\mathcal N}+h
\right).
\end{equation}
Thus, the deterministic validity condition
\eqref{eq:deterministic_formulation_two_condition_equivalent}
holds. Therefore, by
Proposition~\ref{prop:formulation_two_deterministic},
the corrected model is a valid lower bound on $Z$.
\end{proof}

\subsection{Proof of Theorem~\ref{thm:probabilistic_formulation_two}}
\label{app:probabilistic_formulation_two}

\begin{proof}
Under the geometric acceptance condition $\mathcal G$, if
$h\leq h_\omega$, then
Lemma~\ref{lem:sufficient_fill_distance} gives
\[
f^{\mathrm{lb}}(z)\leq f(z),
\qquad
z\in Z.
\]
Hence,
\begin{equation}
\mathcal G
\cap
\left\{
\exists z\in Z:
f^{\mathrm{lb}}(z)>f(z)
\right\}
\subseteq
\left\{
h>h_\omega
\right\}.
\label{eq:app_invalid_implies_coverage_failure}
\end{equation}
Therefore,
\begin{align}
\Pr
\left(
\mathcal G
\cap
\left\{
\exists z\in Z:
f^{\mathrm{lb}}(z)>f(z)
\right\}
\right)
&\leq
\Pr\!\left(h>h_\omega\right)
\\
&\leq
J_\omega^{d_{s_k}}
\left(
1-\frac{1}{J_\omega^{d_{s_k}}}
\right)^{N_s},
\end{align}
where the second inequality follows from
\eqref{eq:coverage_failure_probability}.

Since probabilities are bounded above by one,
\begin{equation}
\Pr
\left(
\mathcal G
\cap
\left\{
\exists z\in Z:
f^{\mathrm{lb}}(z)>f(z)
\right\}
\right)
\leq
\min
\left\{
1,\,
J_\omega^{d_{s_k}}
\left(
1-\frac{1}{J_\omega^{d_{s_k}}}
\right)^{N_s}
\right\}.
\end{equation}
By the definition of
$\delta(\omega,N_s)$ in
\eqref{eq:formulation_two_failure_probability},
the result follows.
\end{proof}

\section{CUTEst Benchmark Problems}
\label{app:benchmark_problems}

The four benchmark problems considered in Section~\ref{sec:benchmark_examples}
are scalable problems selected from the CUTEst test set~\cite{gould2015cutest}.
Each problem admits a star-structured decomposition with a one-dimensional
coupling-variable space. Table~\ref{tab:benchmark_problem_summary} summarizes
their objective functions, coupling and private variables, and convexity
properties. 

\begin{table*}[htbp]
\centering
\caption{Summary of the structured CUTEst benchmark problems.}
\label{tab:benchmark_problem_summary}

\begin{adjustbox}{max width=\textwidth}
\begin{tabular}{ccccc}
\toprule
Problem
& Objective function
& Coupling variable
& Private variables
& Convexity \\
\midrule

ARWHEAD
&
$\displaystyle
\sum_{i=1}^{n-1}
\left[
\left(x_i^2+x_n^2\right)^2-4x_i+3
\right]
$
&
$x_n$
&
$x_i,\; i=1,\ldots,n-1$
&
Convex
\\[3ex]

LIARWHD
&
$\displaystyle
4\sum_{i=1}^{n}\left(x_i^2-x_1\right)^2
+
\sum_{i=1}^{n}\left(x_i-1\right)^2
$
&
$x_1$
&
$x_i,\; i=2,\ldots,n$
&
Nonconvex
\\[3ex]

NONDIA
&
$\displaystyle
\sum_{i=2}^{n}
\left[
100\left(x_1-x_i^2\right)^2
+
\left(1-x_i\right)^2
\right]
$
&
$x_1$
&
$x_i,\; i=2,\ldots,n$
&
Nonconvex
\\[3ex]

EG2
&
$\displaystyle
\sum_{i=1}^{n-1}
\sin\left(x_1+x_i^2-1\right)
+
\frac{1}{2}\sin\left(x_n^2\right)
$
&
$x_1$
&
$x_i,\; i=2,\ldots,n$
&
Nonconvex
\\

\bottomrule
\end{tabular}
\end{adjustbox}

\end{table*}

\vspace{7mm}

\bmhead{Funding}
This research was supported by the UNC System Research Opportunities Initiative (ROI) and National Science Foundation (NSF) Award CBET-2543029.
\bmhead{Acknowledgments}
The authors gratefully acknowledge the computing resources provided by the North Carolina State University High Performance Computing Services Core Facility (RRID:SCR\_022168).
\bmhead{Data Availability Statement}
Codes are available at our GitHub repository: \url{https://github.com/dfasiku/distributed_data_driven_branch_and_bound}, from which all data associated with the numerical studies in this paper can be reproduced. 

\bmhead{Declarations}
The authors declare that they have no conflict of interest.

\bibliography{sn-bibliography}

\end{document}